\documentclass[11pt]{article}
\usepackage{latexsym,amsthm,verbatim,ifthen,graphicx,color,mathrsfs}
\usepackage[all]{xy}
\usepackage{color}

\usepackage{color}

\usepackage{lscape}
 \usepackage[ansinew]{inputenc}
 \usepackage{multicol}
 \usepackage[all]{xy}\xyoption{poly}
 \usepackage{mathrsfs}
\usepackage[psamsfonts]{eucal}
\usepackage{amssymb, amsmath}
\usepackage{geometry}
\usepackage{enumerate}
\usepackage{float}

\newtheorem{theorem}{Theorem}[section]

\newtheorem{corollary}[theorem]{Corollary}
 \newtheorem{lemma}[theorem]{Lemma}
 \newtheorem{proposition}[theorem]{Proposition}
 \theoremstyle{definition}
 \newtheorem{definition}[theorem]{Definition}

 \newtheorem{rem}[theorem]{Remark}
 \newtheorem{ex}[theorem]{Example}

\def\cC{{\mathcal C}}

\def\cG{{\mathscr G}}
\def\gauge{\,\cG\,}

\def\dtil{{\widetilde D}}

\def\quil{{\mathscr L}}
\def\quilc{\widehat{\mathscr L}}
\def\quic{{\mathscr C}}

\def\esp#1{{\quad\text{#1}\quad}}

\newcommand{\bz}{\mathbb Z}
\newcommand{\bulletchi}{{\scriptstyle\,\bullet\,}}
\def\bq{\mathbb{Q}}
\def\br{\mathbb{R}}

\def\Q{\mathbb{Q}}

\def\im{{\rm Im\,}}

\def\des{{s^{-1}}}
\newcommand{\fib}{\operatorname{{\cal F}\text{\rm ib}}}

\def\timest{\widetilde\times}

    \newcommand{\lasu}{{\mathfrak{L}}}
    
     \newcommand{\der}{{\mathcal{D}er}}
      \newcommand{\derr}{{\rm Der}}
      \newcommand{\derge}{{\rm Der}^{\mathcal G}}
      \newcommand{\derpi}{{\rm Der}^{\Pi}}
      \newcommand{\deri}{{\rm Der}^{\cali}}
       \newcommand{\derko}{{\rm Der}^{\mathcal K}_0}
       \newcommand{\derk}{{\rm Der}^{\mathcal K}}
      \newcommand{\coder}{{\rm Coder}}

\newcommand{\nil}{\operatorname{\text{\rm nil}}}

 \newcommand{\lib }{\mathbb{L}}
  \newcommand{\calge }{{\cal G}}
  \newcommand{\calj }{{\cal J}}
    \newcommand{\calk }{{\cal K}}
    \newcommand{\cali }{{\cal I}}

 \newcommand{\fil}{\operatorname{{\mathcal F}}}
  \newcommand{\gil}{\operatorname{{\cal G}}}

  \newcommand{\coprodc}{\,\widehat{\amalg }\,\,}

  \newcommand{\Hom}{\operatorname{\text{\rm Hom}}}

\newcommand{\catcgrp}{\operatorname{{\bf cgrp}}}
\newcommand{\catcla}{\operatorname{{\bf cl}}}

\newcommand{\Span}{\operatorname{{Span}}}

\newcommand{\catdgl}{\operatorname{{\bf dgl}}}

\newcommand{\catcdgl}{\operatorname{{\bf cdgl}}}
\newcommand{\catcdgle}{\operatorname{{\bf cdgl^{\Delta}}}}

\newcommand{\catss}{\operatorname{{\bf sset}}}

\newcommand{\cale}{{\mathcal E}}

 \newcommand{\map}{\operatorname{{\rm map}}}
  \newcommand{\Map}{\operatorname{{\rm map}}}
    
    \newcommand{\id}{\operatorname{{\rm id}}}
        \newcommand{\ad}{\operatorname{{\rm ad}}}

 \newcommand{\MC}{\operatorname{{\rm MC}}}
  \newcommand{\aut}{\operatorname{{\rm aut}}}

     \newcommand{\ev}{\operatorname{{\rm ev}}}
\newcommand{\mc}{{\MC}}

\newcommand{\catcdgc}{\operatorname{{\bf cdgc}}}

  \newcommand{\otimesc}{\widehat{\otimes}}

   \newcommand{\libc}{{\widehat\lib}}

\begin{document}

\title{Lie models of homotopy automorphism monoids and classifying fibrations}

\author{Yves F\'elix, Mario Fuentes and Aniceto Murillo\footnote{The  authors have been partially supported by the  MICINN grant PID2020-118753GB-I00 of the Spanish Government.}}

\maketitle

\begin{abstract}
Given $X$ a finite nilpotent simplicial set, consider  the  classifying fibrations
$$
X\to B\aut_G^*(X)\to B\aut_G(X)\esp{and} X\to Z\to B\aut_\pi^*(X)
$$
where $G$ and $\pi$ denote, respectively, subgroups of the free and pointed homotopy classes of free and pointed self homotopy equivalences of $X$ which act nilpotently on $H_*(X)$ and $\pi_*(X)$.
 We give algebraic models, in terms of complete differential graded Lie algebras (cdgl's), of the rational homotopy type of these fibrations. Explicitly, if $L$ is a cdgl model of $X$, there are connected sub cdgl's $\derge L$ and $\derpi L$  of the Lie algebra of derivations of $L$ such that the geometrical realization of the sequences of cdgl morphisms
 $$
 L\stackrel{\ad}{\to} \derge L\to \derge L\timest sL\esp{and}L\to L\timest\derpi L\to\derpi L
 $$
 have the rational homotopy type of the above classifying fibrations. Among the consequences we also describe in cdgl terms the Malcev $\bq$-completion of $G$ and $\pi$ together with the rational homotopy type of the classifying spaces $BG $ and $B\pi$.

\end{abstract}

\section*{Introduction}

\bigskip

A foundational result of Stasheff \cite{sta} asserts that the classifying space $B\aut (X)$  of  the topological monoid $\aut(X)$ of self homotopy equivalences of a given finite complex $X$ classifies fibrations  with fiber $X$ by means of a universal  fibration over $B\aut(X)$. Furthermore, see \cite{may}, certain submonoids of $\aut (X)$ produce also classifying fibrations for distinguished families of fibrations. The universal cover of $B\aut(X)$ has the  homotopy type of the classifying space $B\aut_1(X)$ of the monoid of self homotopy equivalences  homotopic to the identity on $X$. Whenever $X$ is simply connected  (see \cite{SS1,tan},  \cite{ber2} for a remarkable extension to the fiberwise setting or \cite{ber3} for the relative case)  the rational homotopy type of $B\aut_1(X)$ is well understood and described in terms of classical algebraic models of $X$. However, the homotopy type of $B\aut (X)$ is as unmanageable, even for $X$ simply connected,   as  its fundamental group $\cale(X)$ of homotopy classes of self homotopy equivalences. For instance, as we will see, the homotopy type of $B\aut(S^n_\bq)$  cannot be obtained as the geometrical realization of any known algebraic structure shaping the rational homotopy type of spaces.

Nevertheless, in this paper we are able to algebraically describe the rational homotopy type of the classifying space of certain non-connected distinguished submonoids of $\aut (X)$. Furthermore, we model the corresponding universal classifying fibrations. As a consequence we also read in terms of these models the structure of nilpotent subgroups, and their classifying spaces, of free and pointed homotopy classes of self homotopy equivalences.

To do so we strongly rely on the homotopy theory developed in the category   $\catcdgl$ of complete differential graded Lie algebras (cdgl's henceforth) by means of the Quillen pair
\begin{equation}\label{pairintro}
\xymatrix{ \catss& \catcdgl \ar@<1ex>[l]^(.50){\langle\,\cdot\,\rangle}
\ar@<1ex>[l];[]^(.50){\lasu}\\}
\end{equation}
given by the global model and realization functor (see next section for a quick review or \cite{bfmt0} for a detailed presentation).

From now on, and throughout the paper, we will often not distinguish a simplicial set $X$ from its realization as a CW-complex. Such an object will always be pointed and we denote by $\aut^*(X)$ the submonoid of $\aut(X)$ of pointed self homotopy equivalences.

Also, by a fibration sequence $F\to E\to B$ we understand its more general meaning: the composition of the two maps is homotopic to a constant $b$ and the induced map from $F$ to the homotopy fiber of $E\to B$ over $b$ is a weak homotopy equivalence. This also applies to the category $\catcdgl$.

 Let $X$ be a finite connected complex, let $G\subset\cale(X)$ be a subgroup,  and consider $\aut_G(X)\subset\aut(X)$ the submonoid of self homotopy equivalences whose homotopy classes are elements in $G$. Then, evaluating each equivalence at the base point induces a fibration sequence
$$
 X\longrightarrow B\aut_G^*(X)\longrightarrow B\aut_G(X)
 $$
 which classifies fibrations $X\to E\to B$ where the image of the holonomy action $\pi_1(X)\to \cale(X)$ lies in $G$. Here, $\aut_G^*(X)\subset\aut^*(X)$ is the submonoid of pointed homotopy  equivalences whose (free) homotopy classes lie in $G$.

 Analogously, given $\pi\subset \cale^*(X)$ a subgroup of pointed homotopy classes of pointed self equivalences, there is a fibration sequence  of pointed spaces endowed with a homotopy section (pointed fibration from now on)
 $$
 X\longrightarrow Z\longrightarrow B\aut_\pi^*(X)
 $$
  which classifies pointed fibrations $X\to E\to B$ where the image of the holonomy action, this time understood as $\pi_1(X)\to \cale^*(X)$, lies in $\pi$.

 We now assume that $X$ is nilpotent and $G\subset\cale(X)$  acts nilpotently on $H_*(X)$. Take, in $\catcdgl$, the minimal model $L$  of $X$. Then, there is a subgroup $\calge$  of homotopy classes of cdgl automorphisms of $L$ together with an action of $H_0(L)$, endowed with the group structure given by the Baker-Campbell-Hausdorff (BCH) product, such that $ \calge/H_0(L)$ is isomorphic to $G_\bq$, the rationalization of $G$. Consider $\derge L$ the connected sub differential graded Lie algebra of $\derr L$, the derivations of $L$, formed by the $\calge$-derivations of $L$. These are all derivations of positive degrees and, in degree $0$, those derivations $\theta$ whose exponential $e^\theta$ is an automorphism of $L$ in $\calge$. Then, we prove (see \S\ref{princi} for a thorough and precise preamble of this result):

 \begin{theorem}\label{mainre1intro} The rational homotopy type of the classifying fibration sequence
 $$
 X\longrightarrow B\aut_G^*(X)\longrightarrow B\aut_G(X)
 $$
 is modeled by the cdgl fibration sequence
 $$
 L\stackrel{\ad}{\longrightarrow}\derge L\longrightarrow \derge L\timest sL.
 $$
 \end{theorem}
In fact, see Corollary \ref{corofin}, this  cdgl sequence is part of a larger one
$$
\Hom(\overline\quic(L),L)\to\Hom(\quic(L),L)\to L\stackrel{\ad}{\to}\derge L\to\derge L\timest sL,
$$
whose suitable  restriction  to certain components of $\Hom(\quic(L),L)$ and $\Hom(\overline\quic(L),L)$ models the fibration sequence
$$
\aut_{G_\bq}^*(X_\bq)\to\aut_{G_\bq}(X_\bq)\stackrel{\ev}{\to} X_\bq \to B\aut_{G_\bq}^*(X_\bq)\to  B\aut_{G_\bq}(X_\bq)
$$

Whenever $X$ is simply connected, choosing  universal covers in Theorem \ref{mainre1intro} recovers the classical result mentioned above as the pair (\ref{pairintro}) extends, in the homotopy categories, the classical Quillen model and realization functors \cite{qui} in the simply connected setting.

On the other hand, let $\pi\subset \cale^*(X)$ acting nilpotently on $\pi_*(X)$. Then, there is a subgroup $\Pi$ of homotopy classes of automorphisms of $L$ naturally isomorphic to $\pi_\bq$. As before, let $\derpi L$ be the connected dgl of $\Pi$-derivations of $L$. Then (see again \S\ref{princi} for a detailed statement):

 \begin{theorem}\label{mainre2intro} The rational homotopy type of the pointed classifying fibration sequence
 $$
 X\longrightarrow Z\longrightarrow B\aut_\pi^*(X)
 $$
 is modeled by the cdgl fibration sequence
 $$
 L{\longrightarrow}L\timest \derpi L\longrightarrow \derpi L.
 $$
 \end{theorem}

Under the assumptions and notation of the above results, immediate consequences are:

\begin{corollary}\label{corounointro}$\pi_\bq\cong H_0(\derpi L)$ and $G_\bq\cong H_0(\derge L)/ \im H_0(\ad)$.
\end{corollary}
In both cases $H_0$ is considered as a group with the BCH product.

\begin{corollary}\label{corointrodos} $BG$ and $B\pi$ have the rational homotopy type of the realization of the cdgl's $\derge_0 L\oplus R$ and $\derpi_0 L\oplus S$ respectively.
\end{corollary}
Here, $R$ and $S$ denote a complement of the $1$-cycles of $\derge L\timest sL$ and $\derpi L$ respectively.

\medskip

Along the way to the above results we must translate to $\catcdgl$  the needed geometric phenomena in each situation, mostly arising from the non triviality of the fundamental group.  Distinguishing free and pointed homotopy classes from the cdgl point of view, or describing evaluation fibrations as cdgl's morphisms of derivations, are examples of this  conveyance which enlarges the dictionary between the homotopy categories of simplicial sets and cdgl's. This is precisely and briefly outlined in the following  summary of the content.

   In \S\ref{prelimi} we give a short digest of the main facts concerning the homotopy theory of cdgl's and its connection with that of simplicial sets by means of the pair (\ref{pairintro}). From this point on, we refer to this part for general notation in this matter.

   In the second section we describe the subtle but important difference between free and pointed homotopy classes in $\catcdgl$: given cdgl's $L$ and $L'$, the group $H_0(L)$ endowed with the BCH product acts on the set of cdgl homotopy classes $\{L',L\}$ by means of the exponential map $[x]\bullet[\varphi]=e^{\ad_x}\circ \varphi$. Whenever $L$ is connected and  $L'$ is the minimal model of the connected simplicial set $X$, $\{L',L\}$ is in bijective correspondence with the set of pointed homotopy classes $\{X,\langle L\rangle\}^*$ (Corollary \ref{corouno}) and the mentioned action corresponds to that of  $\pi_1(X)$ on  $\{X,\langle L\rangle\}^*$ (Theorem \ref{modeloaccion}). That is, $\{X,\langle L\rangle\}\cong\{L',L\}/H_0(L)$.

   Again by the presence of non trivial fundamental groups, we need to analyze distinguished subgroups of degree $0$ elements of particular cdgl's, as always endowed with the BCH product. To this end we recall  in Section \ref{cs} the isomorphism between the category of (ungraded) complete Lie algebras and (Malcev) complete groups and prove  in Theorem \ref{cerrado} a useful characterization of these groups. Also in this section we recall the formal  definition of the logarithm and exponential for a given (ungraded) complete Lie algebra. Via the exponential we observe that, whenever $M$ is a certain sub lie algebra of $\derr_0 L$ for some cdgl $L$, then the group $(M,*)$ can be naturally identified with a subgroup of $\aut (L)$.

   In  \S\ref{evalufibra} we describe cdgl models of evaluation fibrations $\map^*(X,Y)\to\map(X,Y)\to Y$ (Theorem \ref{modelohom}) and any of its components (Proposition \ref{escapa}) by refining existing Lie models of mapping spaces. Later on, in Section \ref{deriva} we transform these models in terms of derivations of the models of the involved simplicial sets (Theorem \ref{mappingder} and Corollary \ref{cor4}). Of special interest is the description of the long homotopy exact sequence of the evaluation fibration as the long  homology sequence of a short exact sequence of chain complexes of derivations (Theorem \ref{corhom}).

   Although $\catcdgl$ is complete and cocomplete from the categorical point of view, essential manipulations of cdgl's may lie outside of this category. The twisted product $L\timest M$ of cdgl's, or even the dgl $\derr L$ of derivations of a given Lie algebra are examples of this malfunction, as they fail to be complete even if $L$ and $M$ are. In \S\ref{comple} we fix these issues for distinguished sub dgl's of $\derr L$ (Proposition \ref{derkacomp}) and the particular twisted products we use (Propositions \ref{homcomp}, \ref{twistcomp} and \ref{excusa}).

    In Section \ref{princi} we present a  detailed and precise statement of our main results, together with their first consequences.

   Section \ref{demostra} contains the proofs of Theorems \ref{mainre1intro} and \ref{mainre2intro} and finally, in \S\ref{apli}, we present several consequences of these results, of which Corollary \ref{corounointro} is an immediate one.  Also (see Theorem \ref{propopos} and Propositions \ref{recubre1}, \ref{recubre2}), we model the rational homotopy type of the fibration sequences,
   $$
   \aut_G(X)\to G\to B\aut_1(X)\to B\aut_G(X)\to BG
    $$
    and
    $$\aut_\pi^*(X)\to \pi\to B\aut_1^*(X)\to B\aut_\pi^*(X)\to B\pi,
   $$
obtaining in particular  Corollary \ref{corointrodos}. We finish the section with two sets of examples covering a wide range: every finitely generated  rational nilpotent group is first easily realized as a subgroup of self homotopy equivalences acting nilpotently either in the (rational) homology or homotopy groups of some suitable complexes. We then apply our main results to find explicit Lie models of derivations for the corresponding classifying spaces.

\bigskip

\noindent{\bf Acknowledgement.}  We  thank Alexander Berglund for helpful conversations. We are also extremely grateful to the referee. His/her  numerous suggestions and corrections  have prevented some flaws in this article, while greatly improving it.

\section{Homotopy theory of complete Lie algebras}\label{prelimi}
In this section we recall the basics for complete differential graded Lie algebras, from the homotopy point of view. For it, original references are \cite{bfmt4,bfmt2,bfmt3,bfmt1} whose main results are developed in the complete and detailed monograph \cite{bfmt0}. Sometimes we will also use classical facts from the Sullivan commutative approach to rational homotopy theory. These are not included here and we refer to  \cite{FHT, FHTII} as standard references.

For any category $\cC$, we denote by $\Hom_\cC$ its morphisms, except for that of (graded) vector spaces whose morphisms  will be denoted by the unadorned $\Hom$.

All considered chain (or cochain) complexes, with possibly extra structures, are unbounded and have $\bq$ as ground field. The {\em suspension}  of  such a complex $V$ is denoted by $sV$ where $(sV)_n=V_{n-1}$, $n\in\bz$ and $dsv=-sdv$. The {\em desuspension} complex is denoted by $\des V$.

We denote by $\catss$ the category of simplicial sets. For any $n\ge 0$, we denote by  $\underline\Delta^n=\{\underline\Delta^n_p\}_{p\ge 0}$\label{gunderlineDelta^n} the simplicial set in which $
\underline\Delta^n_p=\{(j_0,\dots,j_p)\mid 0\le j_0\le \dots\le j_p\leq n\}
$ with the usual faces and degeneracies. Given $X$ and $Y$ simplicial sets we denote by $\map(X,Y)$ the {\em simplicial mapping space}, that is,
$$
\map(X,Y)_\bullet=\Hom_{\catss}(X\times\underline\Delta^\bullet,Y).
$$
Analogously, we denote by $\map^*(X,Y)$ the {\em pointed mapping space} $\Hom_{\catss^*}(X\times\underline\Delta^\bullet,Y)$ in the category $\catss^*$ of pointed simplicial sets.

Let $\catdgl$ denote the category of  differential graded Lie algebras (dgl's henceforth). A dgl $L$, or $(L,d)$ if we want to specify the differential, is  {\em $n$-connected}, for $n\ge 0$, if $L=L_{\ge n}$. As usual, connected means $0$-connected.  The $n$-connected cover of a dgl $L$ is the $n$-connected sub dgl ${ L}^{(n)}$ given by
$$
L^{(n)}_p=\begin{cases} \ker d&\text{if $p=n$},\\ \,\,\,L_p&\text{if $p>n$}.\end{cases}
$$
We denote by $\widetilde L=L^{(1)}$ the $1$-connected cover.

 Given a dgl $L$ we denote by $\derr L$ the  dgl of its {\em derivations} in which, for each $n\in\bz$, $\derr_n L$ are linear maps $\theta\colon L\to L$ of degree $n$  such that
$$
\theta[a,b]=[\theta(a),b]+(-1)^{|a|n}[a,\theta(b)],\esp{for} a,b\in L.
$$
The Lie bracket and differential are given by,
$$
[\theta,\eta]=\theta\circ\eta-(-1)^{|\theta||\eta|}\eta\circ\theta,
\quad
D\theta=d\circ\theta-(-1)^{|\theta|}\theta\circ d.
$$
On the other hand, given a dgl morphism $f\colon L\to M$ we denote by $\derr_f(L,M)$ the chain complex of {\em $f$-derivations} in which $\derr_f(L,M)_n$ are linear maps $\theta\colon L\to M$ of degree $n$ such that
$$\theta [a,b] = [\theta (a), f(b)] + (-1)^{n  \vert a \vert} [f(a), \theta (b)],\quad a,b\in L.$$
The differential is given as before.

Following \cite[\S7.2]{tan}, or more generally \cite[\S3.5]{ber2}, a {\em twisted product} $L\timest M$ of the dgl's $L$ and $M$ is a dgl structure on the underlying vector space  $L\times M$ so that
$$
0\to L\longrightarrow L\timest M{\longrightarrow} M\to 0
$$
is an exact dgl sequence. In particular $L$ is a sub dgl of the twisted product.

A {\em filtration} of a dgl $L$  is a decreasing sequence of differential Lie ideals,
$$L=F^1\supset\dots \supset F^n\supset F^{n+1}\supset\dots$$ such that $[F^p,F^q]\subset F^{p+q}$ for $p,q\geq 1$. In particular, if
$$
L^1\supset\dots\supset L^n\supset L^{n+1}\supset\dots
$$
denote the lower central series of $L$, that is,  $L^{1}= L$ and $L^{n}= [L, L^{n-1}]$ for $n>1$, then $L^n \subset F^n$ for all $n$.

A {\em complete differential graded Lie algebra}, cdgl henceforth,  is a dgl $L$ equipped with a  filtration $\{F^n\}_{n\ge 1}$ for which the  natural map
$$
L\stackrel{\cong}{\longrightarrow}\varprojlim_n L/F^n
$$
is a dgl isomorphism. A {\em morphism} $f\colon L\to L'$ between cdgl's, associated to filtrations $\{F^n\}_{n\ge 1}$ and $\{{G}^n\}_{n\ge 1}$ respectively, is a dgl morphism which preserves the filtrations, that is, $f(F^n)\subset G^n$ for each $n\ge 1$.
We denote by $\catcdgl$ the corresponding category which is complete and cocomplete \cite[Proposition 3.5]{bfmt0}.

Given a  dgl  $L$  filtered by $\{F^n\}_{n\ge 1}$, its {\em completion} is the dgl
$$
\widehat L=\varprojlim_nL/F^{n}.
$$
which is always complete with respect to the filtration
$$
\widehat{F}^n=\ker ( \widehat L \to L/F^n)
$$
as $ \widehat L/\widehat{F}^n=L/F^n$. Unless a particular filtration is explicitly mentioned, the completion of a given dgl is  taken over the lower central series. This is the case, in particular, for any cdgl $L=(\libc(V),d)$ where
$$
\libc(V)=\varprojlim_n\lib(V)/\lib(V)^n$$
is the completion of the free Lie algebra $\lib(V)$ generated by the graded vector space $V$, see \cite[\S3.2]{bfmt0} for an exhaustive treatment of this cdgl. In this instance,
 $$\widehat  F^n=
 \widehat{\mathbb L}^{\geq n} (V)=  \prod_{q\geq n} \mathbb L^q(V),\esp{for} n\ge 1.
 $$
By an abuse of notation, we write $\widehat F^n=L^n$ so that $L=\varprojlim_n L/L^n$.

If $A$ is a commutative differential graded algebra (cdga henceforth) and $L$ is a cdgl with respect to the filtration $\{F^n\}_{n\ge 1}$, we define their {\em complete tensor product} as the cdgl
$$
A\otimesc L=\varprojlim_n A\otimes (L/F^n)
$$
where the differential and the bracket in $A\otimes (L/F^n)$ are defined as usual by $d(a\otimes x) = da\otimes x + (-1)^{\vert a\vert} a\otimes dx$ and $[a\otimes x, a'\otimes x'] =  (-1)^{\vert a'\vert   \vert x\vert} a a' \otimes [x,x']$.

A {\em Maurer-Cartan} element, or simply MC element, of a given dgl $L$ is an element $a\in L_{-1}$ satisfying the Maurer-Cartan equation $da=-\frac{1}{2}[a,a]$. We denote by $\mc(L)$ the set of MC elements in $L$. Given $a\in \mc(L)$, we denote by $d_a= d+\ad_a$ the {\em perturbed differential}. This is a new differential on $L$ where $d$ is the original one and $\ad$ is the usual adjoint operator. The {\em component} of $L$ at $a$ is the connected sub dgl $L^a$ of $(L,d_a)$ given by
$$
L^a_p=\begin{cases} \ker d_a&\text{if $p=0$},\\ \,\,\,L_p&\text{if $p>0$}.\end{cases}
$$
In other terms, $L^a=(L,d_a)^0$.
 Whenever $L$ is complete the group $L_0$, endowed  with the Baker-Campbell-Hausdorff product, acts  on the set  $\mc(L)$ by
$$
x\,\cG\, a=e^{\ad_x}(a)-\frac{e^{\ad_x}-1}{\ad_x}(d x)= \sum_{i\geq 0} \frac{\ad_x^i(a)}{i!} - \sum_{i\geq 0} \frac{\ad_x^i(dx)}{(i+1)!},\quad x\in L_0,\,a\in\mc(L).
$$
This is the {\em gauge} action and we denote by $\widetilde\mc(L)=\mc(L)/\cG$ the corresponding orbit set. This is thoroughly studied in  \cite[\S4.3]{bfmt0} and a particular homotopically friendly description of the gauge action can be found in  \cite[\S5.3]{bfmt0}.

There is a pair of adjoint functors, {\em (global) model} and {\em realization},
\begin{equation}\label{pair}
\xymatrix{ \catss& \catcdgl \ar@<1ex>[l]^(.50){\langle\,\cdot\,\rangle}
\ar@<1ex>[l];[]^(.50){\lasu}\\},
\end{equation}
which are deeply studied in \cite[Chapter 7]{bfmt0} and are based in the cosimplicial cdgl
$$\lasu_\bullet=\{\lasu_n\}_{n\ge 0}.
$$
 For each $n\ge 0$, see \cite[Chapter 6]{bfmt0},
$$
\lasu_n=\bigl(\libc(s^{-1}\Delta^{n}),d)
$$
where
$s^{-1}\Delta^{n}$  denotes the desuspension $s^{-1}N_*(\underline\Delta^{n})$ of the non-degenerate simplicial chains  on $\underline\Delta^n$. That is, for any $p\ge 0$, a generator of  degree $p-1$ of $s^{-1}\Delta^n$ can be written as $a_{i_0\dots i_p}$ with $0\le i_0<\dots<i_p\le n$. The cofaces and codegeneracies in $\lasu_\bullet$ are induced by those on the cosimplicial chain complex $\des N_*\underline\Delta^\bullet$, and  the differential $d$ on each $\lasu_n$ is the only one (up to cdgl isomorphism) satisfying:

 \begin{itemize}
 \item[(1)] For each $i=0,\dots,n$, the generators of $s^{-1}\Delta^{n}$, corresponding to vertices, are MC elements.

     \item[(2)] The linear part of $d$ is induced by the boundary operator of $s^{-1}\Delta^{n}$.
     \item[(3)] The cofaces and codegeneracies are cdgl morphisms.
     \end{itemize}
     In particular, $\lasu_0$ is the free Lie algebra $\lib(a)$ generated by a MC element and $\lasu_1$ is the {\em Lawrence-Sullivan interval} \cite{LS} (see also \cite{BM2}).

     The realization of a given cdgl $L$ is the simplicial set
     $$
     \langle L\rangle=\Hom_{\catcdgl}(\lasu_\bullet,L).
     $$
     In particular, the set $\langle L\rangle_0$ of $0$-simplices coincides with $\mc(L)$. Moreover,  see \cite[\S7.2]{bfmt0}, if $\langle L\rangle^a$ denotes  the path component of $\langle L\rangle$ containing $a$, we have:
     \begin{equation}\label{realizasimp}
     \langle L\rangle^a\simeq \langle L^a\rangle,\quad \langle L\rangle\simeq  \amalg_{a\in \widetilde{\mc}(L)} \,\, \langle L^a\rangle.
\end{equation}
      However, it is important to observe that the realization of a cdgl is invariant under perturbations. That is, for any cdgl $L$ and any $a\in \mc(L)$,
  \begin{equation}\label{realizaigual}
     \langle L\rangle\cong\langle (L,d_a)\rangle.
   \end{equation}
   Indeed, by \cite[Proposition 4.28]{bfmt0} there is a bijection $\widetilde\mc(L)\cong\widetilde\mc(L,d_a)$ which sends the gauge class of an MC element $z$ of $L$ to the gauge class of the MC element $z-a$ of $(L,d_a)$. On the other hand, observe that $L^z=(L,d_a)^{a-z}$. Finally, apply (\ref{realizasimp}).

     Also, if $L$ is connected, and for any $n\ge 1$, the map
     $$
     \rho_n\colon \pi_n\langle L\rangle\stackrel{\cong}{\longrightarrow} H_{n-1}(L),\quad \rho[\varphi]=[\varphi(a_{0\dots n})],
     $$
    is a group isomorphism \cite[Theorem 7.18]{bfmt0}. Here, the group structure in $H_0(L)$ is considered with the Baker-Campbell-Hausdorff  product (BCH product henceforth). We also remark that, for each $n\ge 1$, the $n$th connected cover $\langle L\rangle^{(n)}$ of the realization of a connected cdgl is homotopically equivalent to the realization $\langle L^{(n)}\rangle$ of its $n$th connected cover \cite[\S12.5.3]{bfmt0}.

    A final and fundamental property of the realization functor is that it coincides with any known geometrical realization of dgl's: if $L$ is a $1$-connected dgl of finite type  then $\langle L\rangle\simeq \langle L\rangle^Q $, see \cite[Corollary 11.17]{bfmt0}, where $\langle\,\cdot\,\rangle^Q$ stands for the classical Quillen  realization functor \cite{qui}. Furthermore, see \cite[Theorem 0.1]{bfmt4}, \cite[Theorem 3.2]{ni} and \cite[Theorem 11.13]{bfmt0}, for any cdgl $L$, $ \langle L \rangle$ is a strong deformation retract of  $\mc_\bullet(L)$, the Deligne-Getzler-Hinich simplicial realization of $L$ \cite{getz,hi}.

     On the other hand, the global model of a simplicial set $X$ is the cdgl
     $$
     \lasu_X=\varinjlim_{\sigma\in X}\lasu_{|\sigma|}.
     $$
 It can be checked, see \cite[Proposition 7.8]{bfmt0}, that  as complete Lie algebra,   $
\lasu_X=\libc(s^{-1}X)
$ where  $s^{-1}X$ denotes the desuspension $s^{-1}N_*(X)$ of the chain complex of non-degenerate simplicial chains on $X$. Moreover, the differential $d$ on $\lasu_X$  is completely determined by the following:
\begin{itemize}
\item[(1)]
 The  $0$-simplices of $X$ are Maurer-Cartan elements.
\item[(2)]
 The linear part of $d$ is the desuspension of the differential in $N_*(X)$.
\item[(3)]
If $j\colon Y\subset X$ is a subsimplicial set, then
$
\lasu_j=\libc\bigl(s^{-1}N_*(j)\bigr)$.
\end{itemize}

     If $X$ is a simply connected simplicial set of finite type and $a$ is any of its vertices, then \cite[Theorem 10.2]{bfmt0}, $\lasu_X^a$ is quasi-isomorphic to $\lambda (X)$ where $\lambda$ is the classical Quillen dgl model functor \cite{qui}. Moreover, see \cite[Theorem 11.14]{bfmt0}, for any connected simplicial set $X$ of finite type, $
     \langle \lasu_X^a\rangle$ is weakly homotopy equivalent to $\bq_\infty X$ the Bousfield-Kan $\bq$-completion of $X$ \cite{BK}. Recall that, whenever $X$ is nilpotent, its $\bq$-completion coincides, up to homotopy, with its classical rationalization, which we denote by $X_\bq$.

The category $\catcdgl$ has a {\em cofibrantly generated model structure}, see \cite[Chapter 8]{bfmt0}, for which the functors in (\ref{pair}) become a Quillen pair with the classical model structure on $\catss$: a cdgl morphism $f\colon L\to M$ is a fibration if it is surjective in non-negative degrees and is a weak equivalence if $\widetilde\mc(f)\colon\widetilde\mc(L)\stackrel{\cong}{\to} \widetilde\mc(M)$ is a bijection, and $f^a\colon L^a\stackrel{\simeq}{\to} M^a$ is a quasi-isomorphism for every $a\in \mc(L)$.

Although quasi-isomorphisms  are weak equivalences only in the connected case, the {\em generalized Goldman--Millson Theorem}, see  \cite[Theorem 4.33]{bfmt0}, provides a criterion for a quasi-isomorphism to be a weak equivalence in the general case: let $f\colon L\to L'$ be a cdgl morphism   such that map induced by the corresponding filtrations ${F}^n/F^{n+1}\xrightarrow[]{\simeq} {G}^n/G^{n+1}$ is a quasi-isomorphism for  any $n\ge 1$. Then, $f$ is a weak equivalence.

A quasi-isomorphism of connected cdgl's of the form
$$
(\libc(V),d)\stackrel{\simeq}{\longrightarrow} L
$$
makes of  $(\libc(V),d)$ a cofibrant replacement of $L$ and we say that it is a {\em  model of $L$}. If $d$ has no linear term we say that $(\libc(V),d)$ is {\em minimal} and is unique up to cdgl isomorphism. A fundamental object in this theory is:

\begin{definition}\label{minimamodel} Let $X$ be a connected simplicial set and $a$ any of its vertices. The minimal model $(\libc(V),d)$ of  $\lasu_X^a$ is called the {\em minimal model of $X$} \cite[Definition 8.32]{bfmt0}. In the same way, a {\em Lie model of $X$} is a model (non necessarily minimal) of $\lasu_X^a$.
\end{definition}

If $(\libc(V),d)$ is the minimal model of $X$ then, see \cite[Proposition 8.35]{bfmt0}, $sV\cong \widetilde H_*(X;\bq)$ and, provided $X$ of finite type, $sH_*(\libc(V),d)\cong\pi_*(\bq_\infty X)$. Here, the group $H_0(\libc(V),d)$ is again considered  with the BCH product. Furthermore, from all of the above, if $X$ is simply connected, the minimal  model of $X$ is isomorphic to its classical Quillen minimal model.

\bigskip
Let $\catcdgc$ denote the category of cocommutative differential graded coalgebras (cdgc's from now on).  Every cdgc $C$  is always assumed to have a counit $\varepsilon\colon C\to\bq$ and a coaugmentation $\eta\colon\bq\to C$. We write $\overline{C}= \ker\varepsilon$, so that $C\cong\overline C\oplus \bq$. The map $\overline{\Delta}\colon \overline C\to \overline C\otimes\overline C$ is defined by
$
\overline{\Delta}x= \Delta x- (u\otimes x + x\otimes u)
$,
with $u=\eta(1)$. A cdgc $C$ is $n$-connected, for $n\ge 0$, if $\overline C=C_{> n}$.

Recall the classical pair of adjoint functors
\begin{equation}\label{quillenlc}
\xymatrix{
{\rm \bf cdgc} \ar@<0.75ex>[r]^-{\mathscr L} &\mathbf{dgl} \ar@<0.75ex>[l]^(0.49){{\mathscr C} }
}
\end{equation}
defined as follows:

Given $C$ a cdgc,
${\mathscr L}(C)$ is $(\mathbb L(s^{-1}\overline{C}),d)$, the free dgl generated by the desuspension of $\overline C$ and  $d=d_1+d_2$ where
$$
\begin{aligned}
d_1(s^{-1}c) &= -s^{-1}dc,\\
d_2(s^{-1}c) &= \frac{1}{2} \sum_i (-1)^{\vert a_i\vert} [s^{-1}a_i, s^{-1}b_i]\quad\text{with} \quad\overline{\Delta}c = \sum_i a_i\otimes b_i.
\end{aligned}
$$
On the other hand,
${\mathscr C}(L)$ is $(\land (sL),d)$, the free cocommutative coalgebra cogenerated by the suspension of $L$ and
in which $d=d_1+d_2$ where
$$
\begin{aligned}
d_1(sv_1\land \dots \land sv_n)& = -\sum_{i=1}^n(-1)^{n_i} sv_1\land \dots \land s(dv_i)\land \dots \land sv_n,\\
d_2(sv_1\land \dots \land sv_n) &= \sum_{1\leq i<j\leq n} (-1)^{\vert sv_i\vert} \, \rho_{ij}\, s[v_i, v_j]\land sv_1\land \dots  \land \widehat{sv_i}\land \dots \land \widehat{sv_j}\land  \dots \land sv_n,
\end{aligned}
$$
$n_i = \sum_{j<i} \vert sv_j\vert$  and
 $\rho_{ij}$ is the  Koszul sign of the permutation
$$sv_1\land \dots \land sv_n \mapsto sv_i\land sv_j\land sv_1 \land \dots \land \widehat{sv_i}\land \dots \land \widehat{sv_j}\land \dots \land sv_n.$$
In particular,
$d_1(sv) = -sdv$ and $d_2(sv\land sw)= (-1)^{\vert sv\vert} s[v,w]$.

In general, $\quic$ preserves quasi-isomorphisms \cite[Proposition 4.4]{nei}. However, $\quil$ preserves quasi-isomorphisms between $1$-connected cdgc's \cite[Proposition 6.4]{nei} and only between connected fibrant cdgc's of finite type \cite[Proposition 2.4(2)]{bfmt0}.

The \emph{cochain algebra} $\quic^*(L)$ on the dgl $L$ is the cdga dual of  ${\mathscr C}(L)$.
Whenever
$L$ connected and of finite type, this cdga has a well known description:
$$
{\mathscr C}^* (L)\cong(\land (sL)^\#,d)
$$
in which $d=d_1+d_2$ where
$$
\begin{aligned}
\langle d_1v, sx\rangle &= (-1)^{\vert v\vert} \langle v, sdx\rangle,\\
\langle d_2v,sx\land sy\rangle & = (-1)^{\vert y\vert + 1} \langle v, s[x,y]\rangle.
\end{aligned}
$$
Here, $\langle\,\,,\,\rangle\colon \land (sL)^\sharp \times \land sL\to\bq$ denotes the usual pairing, see for instance  \cite[\S23]{FHT}.

The adjunction map
$
 \alpha {\colon} {\mathscr L}{\mathscr C} (L)\to L$ is the unique dgl morphism $(\lib(s^{-1}{\land^+sL}),d)\to L$ extending the projection
$s^{-1}{\land^+sL} \to s^{-1}{\land^+sL} / \left(s^{-1} \land^{\geq 2} sL\right)  \cong L$. The other adjunction  $ \beta {\colon} C\to {\mathscr C}{\mathscr L}(C)$ is the unique cdgc morphism $C\to (\land s\quil(C),d)$ lifting the inclusion
$\overline{C} = ss^{-1}\overline{C} \subset s\mathbb L(s^{-1}\overline{C})$. By  \cite[Proposition 2.3]{bfmt0} $\alpha$ is always a quasi-isomorphism while $\beta$ is a quasi-isomorphism  if $C$ is connected.

\section{Free and pointed homotopy classes in $\catcdgl$}\label{freepoin}

 Throughout the next sections it will be essential to make a clear distinction between the sets of free and pointed homotopy classes of maps, from the cdgl point of view. This is a subtle but crucial point that we settle here. Note that this discussion is meaningless in the classical Quillen theory as, in the simply connected category, both sets coincide.

 Recall from \cite[\S8.3]{bfmt0} that two cdgl morphisms $\varphi,\psi\colon M\to L$ are homotopic if there is a cdgl morphism $\Phi\colon M\to L\otimesc \land(t,dt)$ such that $
\varepsilon_0\circ\Phi=\varphi$ and $\varepsilon_1\circ\Phi=\psi$.
Here, $\land(t,dt)$ is the free commutative graded algebra generated by the element $t$ of degree $0$ and its differential $dt$. Also, for $i=0,1$, $\varepsilon_i\colon  L\otimesc \land(t,dt)\to L$ is the cdgl morphism obtained by sending $t$ to $i$. We denote by $\{M,L\}$ the set of homotopy classes of cdgl morphisms from $M$ to $L$.

 For the rest of this section  we fix a connected cdgl  $L$ and a connected (non necessarily reduced) simplicial set $X$,  pointed by the  $0$-simplex $b\in X_0$.  Analogously, we denote respectively by  $\{X,\langle L\rangle\}$ and $\{X,\langle L\rangle\}^*$ the set of free and pointed homotopy classes of simplicial maps. Note that, since $\langle L\rangle$ is reduced, being $0\in  \mc(L)$ the only $0$-simplex, every  map $X\to \langle L\rangle$ is pointed. Then:

 \begin{proposition}\label{freepointed}
 $
\{X,\langle L\rangle\}\cong\{\lasu_X,L\}$ and  $ \{X,\langle L\rangle\}^*\cong\{\lasu_X/(b),L\}$.
\end{proposition}

Here $(b)$ denotes the Lie ideal generated by the Maurer-Cartan element $b$.

\begin{proof} As the model and realization functors constitute a Quillen pair, the first identity is obvious \cite[Corollary 8.2(iv)]{bfmt0}. Nevertheless, an explicit description of this identity will lead  to the second one:

The adjunction (\ref{pair}) provides a bijection
$$
\Hom_{\catss}(X,\langle L\rangle)\cong \Hom_{\catcdgl}(\lasu_X,L)
$$
which sends any simplicial map $f\colon X\to\langle L\rangle$ to the cdgl morphism  $\varphi_f\colon \lasu_X\to L$ defined as follows: recall that, for $p\ge 0$, a generator of $\lasu_X$ of degree $p-1$ is identified with a non-degenerate $p$-simplex $\sigma\in X_p$. Then, $\varphi_f(\sigma)=f(\sigma)(a_{0\dots p})$, being $a_{0\dots p}\in s^{-1}\Delta^p$ the  top generator. Note that $\varphi_f(b)=0$ so it induces a map $\overline\varphi_f\colon \lasu_X/(b)\to L$.

Now, let  $H\colon X\times \underline\Delta^1\to Y$ be a homotopy between two simplicial maps $f,g\colon X\to \langle L\rangle$. That is, $H$ is a simplicial map for which $H|_{X\times(0)}=f$ and $H|_{X\times(1)}=g$, being $(0)$ and $(1)$  the subsimplicial sets of $\underline\Delta^1$ generated by the $0$-simplices $0$ and $1$ respectively. Taking into account the bijection
$$
\Hom_{\catss}(X\times \underline\Delta^1,\langle L\rangle)\cong\Hom_{\catss}\bigl(X,\map(\underline\Delta^1,\langle L\rangle)\bigr),
$$
$H$ corresponds to a simplicial map $X\to\map(\underline\Delta^1,\langle L\rangle)$. But, see \cite[Theorem 12.18]{bfmt0} which is precisely  \cite[Theorem 6.6]{Berglund} in our context, there is a homotopy equivalence
$$
\map(\underline\Delta^1,\langle L\rangle)\simeq\langle L\otimesc \land(t,dt)\rangle.
$$
 Hence, we identify $H$ with a map $X\to \langle L\otimesc \land(t,dt)\rangle$ which, again by adjunction, corresponds to a cdgl morphism $\Phi\colon \lasu_X\to L\otimesc \land(t,dt)$. The fact that $H$ is a homotopy between $f$ and $g$ translates to $
\varepsilon_0\circ\Phi=\varphi_f$ and $\varepsilon_1\circ\Phi=\varphi_g
$ so that $\Phi$ is a homotopy between $\varphi_f$ and $\varphi_g$. This shows again that $\{X,\langle L\rangle\}\cong\{\lasu_X,L\}$.

 Now, assume that $H$ is a pointed homotopy. That is, $H$ is constant (to the null and only  $0$-simplex of $\langle L\rangle$) on the subsimplicial set $(b)\times\underline\Delta^1$ of $X\times\underline\Delta^1$.  Then,
 a careful inspection along this identification shows that the homotopy $\Phi$ sends $b$ to $0$ (which is not the case in general) and thus it induces a homotopy $\overline\Phi\colon \lasu_X/(b)\to L\otimesc \land(t,dt)$ between $\overline\varphi_f$ and $\overline\varphi_g$. This proves that $ \{X,\langle L\rangle\}^*\cong\{\lasu_X/(b),L\}$.

     \end{proof}

     \begin{corollary}\label{corouno} If $L'$ is the minimal Lie model of $X$, $\{X,\langle L\rangle\}^*\cong\{L',L\}$. In particular, if $L$ is a Lie model of the finite type simplicial set $Y$, $\{X,\bq_\infty Y\}^*\cong\{L',L\}$.
     \end {corollary}

     \begin{proof} It is known, see \cite[Proposition 8.7]{bfmt0}, that the composition
     $$
\lasu_X^b\stackrel{\simeq}{\longrightarrow}(\lasu_X,d_b)\stackrel{\simeq}{\longrightarrow}\lasu_X/(b)
$$
is not only an injective quasi-isomorphism but a weak equivalence in $\catcdgl$. Hence, $L'$ is also weakly equivalent to $\lasu_X/(b)$ and thus $\{L',L\}\cong\{\lasu_X/(b),L\}\cong \{X,\langle L\rangle\}^*$. To finish, recall that being $L$ a Lie model of $Y$, $\langle L\rangle\simeq\bq_\infty Y$.
\end{proof}

We now translate to $\catcdgl$ the action  of $\pi_1\langle L\rangle $ on $\{X,\langle L\rangle\}^*$. For it, and for any given two connected simplicial sets $X$ and $Y$, the action $\bulletchi$ of $\pi_1(Y)$ on $\{X,Y\}^*$ can be described as follows: Consider the wedge $[a,b]\vee X$ of an interval with $X$, and let   $r{\colon} [a,b]\vee X\to X$ be the retraction  obtained by contracting $[a,b]$ into $b$. As $b\hookrightarrow X$ is a cofibration,  there exists a homotopy equivalence $g{\colon}
  X \to [a,b]\vee X$   such that  $g (b)= a$ and $r\circ g\simeq\id_X$. Let $\widetilde g=c\circ g\colon X\to S^1\vee X$ where $c$ glues $a$ with $b$. Then, given  $\alpha\colon S^1\to Y$ and $f\colon X\to Y$, $[\alpha]^*\bulletchi[f]^*$ is represented by $(\alpha\vee f)\circ\widetilde g$.

\smallskip

We define now an analogous action of $H_0(L)$ on $\{\lasu_X/(b), L\}$. For it, recall from \cite[Proposition 4.10]{bfmt0} that, for any given element $y\in L_0$, the exponential
$$
e^{\ad_y}=\sum_{n\ge0}\frac{\ad_y^n}{n!}
$$ is a well defined automorphism of $L$. In fact, this is a particular instance of a more general setting covered by  Remark \ref{remarkexp} below. Furthermore, if $y*z$ denotes the BCH product of elements $y,z\in L_0$, then $e^{\ad_{y*z}}=e^{\ad_y}\circ e^{\ad_z}$, see \cite[Corollary 4.12]{bfmt0}.

\begin{definition}\label{granaccion} Given any cdgl $L'$,  the group $(L_0,*)$ acts on $\Hom_{\catcdgl}(L',L)$ by $y\bulletchi \varphi=e^{\ad_y}\circ \varphi$.  We denote in the  same way the induced action,
$$
H_0(L)\times\{L',L\}\stackrel{\bullet}{\longrightarrow}\{L',L\},\qquad[y]\bulletchi[\varphi]=[e^{\ad_y}\circ \varphi].
$$
\end{definition}

Then, if we denote by $\gamma\colon \{X,\langle L\rangle\}^*\cong \{\lasu_X/(b),L\}$ the explicit bijection described in Proposition \ref{freepointed}, we prove the following which extends \cite[\S12.5.1]{bfmt0}:

\begin{theorem}\label{modeloaccion} The following diagram commutes,
$$\xymatrix{
\pi_1\langle   L\rangle  \times \{X,\langle L\rangle\}^* \ar[r]^(.56)\bullet \ar[d]_{\rho_1\times \gamma}^\cong & \{X,\langle L\rangle\}^*   \ar[d]_\cong^{\gamma}\\
H_0(L) \times \{\lasu_X/(b),L\} \ar[r]^(.58)\bullet  &  \{\lasu_X/(b),L\}.
}
$$
\end{theorem}

\begin{proof}  Note that, by definition, the following commutes:

  \begin{equation}\label{uno}
  \xymatrix{
\pi_1\langle   L\rangle  \times \{X,\langle L\rangle\}^* \ar[dr]^(.53)\bullet \ar[d]_\cong &  \\
\{S^1\vee X,\langle L\rangle\}^*  \ar[r]_(.58){{\widetilde g}^*}  &  \{X,\langle L\rangle\}^* .
}
\end{equation}
By the naturality of the bijections in Proposition \ref{freepointed} the following also commutes:
\begin{equation}\label{dos}
  \xymatrix{
  \{S^1\vee X,\langle L\rangle\}^*  \ar[r]^(.58){{\widetilde g}^*}\ar[d]_\cong  &  \{X,\langle L\rangle\}^*\ar[d]^\cong  \\
\{\lasu_{S^1\vee X}/(b),L\}  \ar[r]_(.58){ {\lasu_{\widetilde g}}^*}  &  \{\lasu_X/(b),L\}.
}
\end{equation}
Next, as any connected simplicial set is weakly equivalent to a reduced one, we lose  no generality by assuming $X$ reduced with $b$ the only $0$-simplex. Hence we may write,
$$
\lasu_X=(\libc(V\oplus \langle b\rangle),d)\esp{with} V=V_{\ge 0}.
$$
Then, \cite[Proposition 7.21]{bfmt0} guarantees that $d$ may be chosen so that, for the perturbed differential,
$$
(\lasu_X,d_b)=(\libc(V),d_b)\coprodc(\lib(b),d_b).
$$
Here, $\coprodc$ denotes the coproduct in $\catcdgl$.
On the other hand, $\lasu_{[a,b]}=\lasu_1$ is the cdgl $(\libc(a,b,x),d)$ in which  $a$ and $b$ are MC elements, and
$$
d x = \ad_xb + \sum_{n=0}^\infty \frac{B_n}{n!}\,  \ad_x^n (b-a)
$$
where $B_n$ denotes the $n$th Bernoulli number. Thus, as the realization functor preserves colimits,
$$
\lasu_{[a,b]\vee X}=(\libc(V\oplus\langle a,b,x\rangle),d).
$$
We define,
$$
h\colon (\lasu_X,d_b)\longrightarrow (\lasu_{[a,b]\vee X},d_a),\qquad h(b)=a,\qquad h(v)=e^{\ad_x}(v),\quad v\in V,
$$
and check that it is a cdgl morphism: obviously $h(d_bb)=d_ah(b)$ while, applying \cite[Proposition 4.24]{bfmt0} and the fact that  $d_b(v)\in\libc(V)$ for any $v\in V$, we get
$$
h(d_bv)=e^{\ad_x}(d_bv)=d_ae^{\ad_x}( v)=d_a h(v).
$$
On the other hand,
$$
\lasu_r\colon(\libc(V\oplus\langle a,b,x\rangle),d)\to (\libc(V\oplus\langle b\rangle),d)
$$
is the identity on $V\oplus\langle b\rangle$, and sends  $a$ to $b$ and $x$ to $0$. Therefore, as $\lasu_r\circ h=\id_{\lasu_X}$, $\lasu_g$ is necessarily homotopic to $h$ and thus, also  up to homotopy,
$$
\lasu_{\widetilde g}\colon (\libc(V\oplus\langle b \rangle),d)\longrightarrow (\libc(V\oplus\langle b,x \rangle),d), \quad \lasu_{\widetilde g}(b)=b, \quad \lasu_{\widetilde g}(v)=e^{\ad_x}(v).
$$
Hence, the induced map on the quotients $
\lasu_{\widetilde g}\colon \lasu_X/(b)\to\lasu_{S^1\vee X}/(b)$ is the morphism
$$
(\libc(V),d)\longrightarrow ( \libc(V),d)\coprodc(\lib(x),0),\quad v\mapsto e^{\ad_x}(v),
$$
and therefore, the bottom row of diagram (\ref{dos}) becomes
$$
\{(\lib(x),0)\coprodc(\libc(V),d),L\}\stackrel {{\lasu_{\widetilde g}}^*}{\longrightarrow}\{(\libc(V),d),L\}.
$$
But $\{(\lib(x),0)\coprodc(\libc(V),d),L\}\cong H_0(L)\times\{\lasu_X/(b)\}$ and thus, the following commutes,
 \begin{equation}\label{tres}
  \xymatrix{
\{\lasu_{S^1\vee X}/(b),L\}  \ar[r]^(.55){ {\lasu_{\widetilde g}}_*}\ar[d]_\cong  &  \{\lasu_X/(b),L\} \\
H_0(L)\times\{\lasu_X/(b)\} \ar[ur]_(.50){\bulletchi}  &
}
\end{equation}
Joining diagrams (\ref{uno}), (\ref{dos}) and (\ref{tres})  produces  the diagram of the statement.
\end{proof}

As $\{X,\langle L\rangle\} \cong \{X,\langle L\rangle\}^*/\pi_1\langle L\rangle$, we deduce:

\begin{corollary}\label{coropointed} If $L'$ is a Lie model of $X$, $\{X,\langle L\rangle\}\cong \{L',L\}/H_0(L)$. \hfill$\square$
\end{corollary}

\begin{ex}
Let $L= (\widehat{\mathbb L} (u,v), 0)$ with $\vert u\vert = \vert v\vert = 0$, which is the minimal model of $S^1\vee S^1$, and recall that the model of the circle is $\lasu_{S^1} = (\widehat{\mathbb L}(b,x),d)$ with $|b|=-1$, $db= -\frac{1}{2}[b,b]$ and $dx= [x,b]$. Denote by $f,g\colon \lasu_{S^1}\to L$ the morphisms defined by $f(x)= v$ and $g(x) =e^{\ad_u}(v)$. The induced morphisms $\overline f,\overline g\colon \lasu_{S^1}/(b)=( \lib(x),0)\to L$ are clearly not homotopic but $u\bulletchi \overline f=\overline g$ which means that $f$ and $g$ are homotopic. An explicit homotopy is given by $\Psi\colon\lasu_{S^1}\to L\otimesc\land(t,dt)$, $\psi(x)= e^{t\, \ad_u}(v)$ and $\Psi(b) = -u dt$. The corresponding maps $S^1\to \bq_\infty(S^1\vee S^1)$ are thus homotopic but not pointed homotopic.
\end{ex}

\section{Complete subgroups of complete Lie algebras}\label{cs}

In this section all complete Lie algebras will be ungraded, or equivalently, concentrated in degree $0$, and with no differential. We denote by $\catcla$ the corresponding category and, to stress this restriction, we will denote by $M$ (instead of $L$) such a general Lie algebra.

 We first recall the category isomorphism between complete Lie algebras and complete groups in the Malcev sense.

 Given a group $G$ we denote its commutator by curved brackets:
$$(x,y)=x y x^{-1}y^{-1}$$ and consider its
lower central series,
$$G=G^1\supset G^2\supset \dots \supset G^i\supset G^{i+1}\supset \dots$$
where  $G^{i}=(G^{i-1},G)$ for $i\geq 2$. A group $G$ is \emph{pronilpotent} if the natural map $G\stackrel{\cong}{\longrightarrow} \varprojlim_n G/G^n$ is an isomorphism. On the other hand recall that $G$ is said to be $0$-local, uniquely divisible or rational if for any natural $n\ge 1$ the map $G\to G$, $ g\mapsto g^n$,
is a bijection.

\begin{definition}\label{malcev}
  A group $G$ is \emph{Malcev $\Q$-complete} (or simply complete) if it is pronilpotent and for each $n\geq 1$ the abelian group
 $G^n/G^{n+1}$ is a $\Q$-vector space. With the obvious notion of morphism we denote by $\catcgrp$ the corresponding category of complete groups.
 \end{definition}

From the original work of Malcev \cite{mal} several reformulations and generalization of the so called {\em Malcev equivalence} can be found in the literature. Following \cite[\S A.3]{qui} and its modern reformulation in \cite[\S8.2.8]{fres}, a  suitable form for our goals of this statement reads:

 There is an isomorphism of categories
$$
\xymatrix{
\catcla \ar@<0.75ex>[r]^-{\Phi} &\mathbf{dgl} \ar@<0.75ex>[l]^(0.49){\Psi }
}
$$
with the following properties:
\begin{itemize}
\item The underlying set is not altered by any of the functors.
\item Given $M$ a  complete Lie algebra $M$,  $\Psi(M)$ is  the group $(M,\ast)$ where $\ast$ is the BCH product. There are also explicit formulas for the Lie algebra structure $\Phi(G)$ for any given complete group $G$ (see for instance \cite[\S 2.3]{ste}).

\item If $\Phi(M)=G$ then for each $n\geq 1$, there is an equality of sets $M^n=G^n$ (see \cite[Theorem 2.2]{FHTII}).

\item When restricted to nilpotent Lie algebras and $0$-local nilpotent groups, this is the original equivalence given in \cite{la} and \cite{mal}.

\end{itemize}

In this context, we will need the following, slightly different, characterization of Malcev complete groups.

\begin{theorem}\label{cerrado}
A group $G$ is Malcev complete if and only if it is pronilpotent and $0$-local.
\end{theorem}
\begin{proof}
The Malcev equivalence implies that a Malcev complete group is equivalent to a Lie algebra. In particular, the multiplication by scalars in the Lie algebra implies that the group is $0$-local.

For the other implication note first that,
 given a complete Lie algebra $M$, the Lie algebras $M^n$ and $M/M^n$ are also complete for any $n\geq 1$. Via the Malcev isomorphism we deduce that $G^n$ and $G/G^n$ are complete groups whenever $G$ is a complete group.

 We only have to prove that $G^n/G^{n+1}$ is  $0$-local for all $n\geq 1$. For it, we
 first show that  the group $G/G^{n}$ is torsion-free for $n\geq 1$. Suppose that there is $x\in G$ such that
$$x^k\equiv 0 \pmod{G^{n}},$$
that is, $x^k\in G^{n}$. We then may write
$$x^k=\prod_{i=1}^{m} (a_i,b_i)$$
where $a_i\in G$ and $b_i\in G^{n-1}$.
Since $G$ is $0$-local, for each $i=1,\dots m$ we can find $c_i\in G$ such that $c_i^k=a_i$. Note that in $G/G^{n+1}$ the class of the  elements of $G^{n}=(G,G^{n-1})$ are in the center and thus one easily checks that
$$
(a_i,b_i)=(c_i^k,b_i)\equiv (c_i,b_i)^k \pmod{ G^{n+1} }.
$$
Then, write
$$y=\prod_{i=1}^m (c_i,b_i)$$
which is a product of elements of $G^{n}$  and thus, modulo $G^{n+1}$, we have:
$$
y^k= \biggl(\prod_{i=1}^m (c_i,b_i)\biggr)^k \equiv \prod_{i=1}^m(c_i,b_i)^k \equiv \prod_{i=1}^m (a_i,b_i)=x^k \pmod{ G^{n+1}}.$$
In particular, since $y\in G^n$,
$$(xy^{-1})^k\equiv 0 \pmod{ G^{n+1}}.$$

Write $x_{n}=x$, $x_{n+1}=xy^{-1}$ and repeat this process to obtain a sequence of elements $(x_j)_{j\geq n}$ in $G$ such that
$$x_{j+1}\equiv x_j \pmod{G^j}$$
and
$x_j^k\in G^j$ for all $j\geq n$.
However, since $G$ is a pronilpotent group, $G\cong\varprojlim G/G^n$ and, denoting $\overline x_j$ the class of $x_j$ in $G/G^j$, we may then identify the sequence $x_0=(\overline x_j)_{j\geq n}$ with an element of $G$ for which
$$x_0\equiv x_i \pmod{G^j}\quad\text{for all $i\geq j$}.
$$
 In particular, $x_0^k\in G^j$ for all $j\geq n$, which implies that $x_0^k=0$. Since $G$ is $0$-local, we deduce that $x_0=0$. Therefore, the sequence
$(\overline x_j)_{j\geq n}$
is identically zero. In particular, $x\in G^n$, and therefore $G/G^n$ is torsion-free.

Next, since $G/G^n$ is nilpotent and torsion-free, by \cite[Theorem 2.2]{HMR}, the map $G/G^n\to G/G^n$,  $\overline x\mapsto \overline x^k$, is injective for any natural $k\ge 1$. On the other hand, it is clearly surjective  as $G$ is $0$-local and we conclude that $G/G^n$ is $0$-local, for any $n\geq 1$.

Finally, consider the short exact sequence
$$G^n/G^{n+1}\to G/G^{n+1} \to G/G^{n}$$
where the two right groups are $0$-local and all of them are nilpotent. Then, \cite[Corollary 2.5]{HMR} implies that $G^n/G^{n+1}$ is also $0$-local.

\end{proof}

 We now briefly  describe the exponential and logarithm maps in for a given complete Lie algebra $M$ and refer to \cite[\S2.4]{FHTII} for details. Given $UM$ the  universal enveloping algebra of $M$ consider  the ideal $I$ of $UM$ generated by $M$. By completing with respect to the filtration given by $I^0=UM$, $I^1=I$
and $I^n=I^{n-1}I$, $n\geq 2$, we get
$$
\widehat{UM} = \varprojlim_{n\ge 0}  UM/I^n \esp{and} \widehat{I} = \varprojlim_{n\ge 1} I/I^n,
$$
together with the well known bijections (see for instance \cite[Chapter 4]{se}), inverses of each other,
$$
\xymatrix{\widehat I\,\, \ar@<0.98ex>[r]_(.43){\scriptscriptstyle \cong}^-{\exp} &{\,1+\widehat I} \ar@<0.95ex>[l]^-{\log} }
$$
given by
$$
\exp(x)=e^x=\sum_{n\ge 0}\frac{x^n}{n!}\esp{and}\log(1+x)=\sum_{n\ge 1}(-1)^{n+1}\frac{x^n}{n}.
$$
Now, the diagonal $\Delta\colon UM\to UM$ clearly sends $I^n$ to $J^n=\sum_{i+j=n}I^i\otimes I^j$ and thus it induces an algebra morphism denoted in the same way
$$
\Delta\colon \widehat{UM}\to \widehat{UM}\otimesc \widehat{UM}=\varprojlim_n (UM\otimes UM)/J^n.
$$
The {\em primitive} and {\em grouplike} elements of $UM$ are  respectively defined by
$$
P=\{x\in\widehat I,\,\,\Delta x=x\otimes 1+1\otimes x\}\,\,\text{and}\,\, G=\{1+y\in 1+ \widehat I,\,\,\Delta(1+y)=(1+y)\otimes (1+y)\}.
$$
Then, $\exp$ and $\log$ restrict to isomorphisms between $P$ and $G$. However, see \cite[Proposition 2.3]{FHTII}, the inclusion $M\to P$ is in fact an isomorphism and thus we have:
$$
\xymatrix{M\,\, \ar@<0.98ex>[r]_(.54){\scriptscriptstyle \cong}^-{\exp} &{\,G} \ar@<0.95ex>[l]^-{\log} }
$$
Now, the multiplication on $\widehat{UM}$ restricts to a product $\cdot$ on $G$ for which it becomes a group. This in turn induces a group structure on $M$ via the classical {\em Baker-Campbell-Hausdorff} product,
$$
x*y=\log(e^x\cdot e^y),\quad x,y\in M.
$$
To finish the section we make  a crucial observation which, via the exponential, let us identify certain complete Lie algebras of $0$-derivations of a given cdgl $L$ with a group of automorphisms of $L$.
\begin{rem}\label{remarkexp}
 Let $L$ be a cdgl with respect to the filtration $\{F^n\}_{n\ge 1}$ and let $M\subset\derr_0 L$ be a sub Lie algebra of $0$-derivations of $L$ whose elements increase the filtration degree. That is, $\theta(F^n)\subset F^{n+1}$ for any $\theta\in M$ and any $n\ge 1$. Note that, in this case, $M$ is a complete Lie algebra with respect to the filtration $\{{\mathcal F}_n\}_{n\ge 1}$ where ${\mathcal F}_n=\{\theta\in M,\,\,\theta(F^m)\subset F^{m+n}\,\text{for all $m$}\}$. Then, the injection
$$
 M\hookrightarrow \Hom(L,L)$$ extends to an algebra morphism,
 $
I\to \Hom(L,L)$
where the product on the endomorphisms of $L$ is given by composition. More generally, by our assumption on $M$,  the above injection induces, also by composition on the right hand side, injections of the form
$$ I^n\hookrightarrow \Hom(L,F^n),\quad n\ge 1,
$$
which define a map
\begin{equation}\label{rev1}
 \widehat I=\varprojlim_{n\ge 1} I/I^n\longrightarrow   \varprojlim_{n\ge 1} \Hom(L,L)/\Hom(L,F^n).
\end{equation}
However, note that
$$ \varprojlim_{n\ge 1} \Hom(L,L)/\Hom(L,F^n)\cong \varprojlim_{n\ge 1}\Hom(L,L/F^n)\cong
 \Hom(L,\varprojlim_{n\ge 1}L/F^n)\cong \Hom(L,L)
 $$
 and thus, (\ref{rev1}) becomes a map
 $
 \widehat I\longrightarrow \Hom(L,L)
 $
which we extend to $$
\xi\colon 1+\widehat I\longrightarrow \Hom(L,L)
$$
by sending $1$ to $\id_L$.
Observe that, by restricting $\xi$ to the grouplike elements  we obtain an injection of  $G$ into the linear (non necessarily compatible with the differential)  automorphisms of $L$. Then, given $\theta\in M$, we abuse of notation and write $e^\theta=\xi\circ\exp(\theta)$, that is,
$$
e^{\theta}=\sum_{n\ge 0}\frac{\theta^n}{n!}
$$
where the product is now composition. This is then a well defined  linear automorphism of $L$. It is  easy to check that it commutes with the Lie bracket in general, and with the differential whenever $\theta$ is a cycle, see for instance \cite[Proposition 4.10]{bfmt0}.

Thus, if $M\subset\derr_0 L\cap\ker D$, the group $(M,*)$, with the BCH product, is identified, via the exponential, with a subgroup of $\aut(L)$.
\end{rem}

\section{Evaluation fibrations}\label{evalufibra}

Given $X$ and $Y$ simplicial sets, the {\em evaluation fibration} is given by
$$
\map^*(X,Y)\longrightarrow\map(X,Y)\stackrel{\ev}{\longrightarrow}Y
$$
where $\ev$ denotes the evaluation at the base point. We will need a particular version of certain results, some of them already known under a different perspective, describing cdgl models of evaluation fibrations.

\medskip

Given any cdgc $C$ and any dgl $L$, consider the dgl structure on $\Hom(C,L)$ with the usual differential $Df=d\circ f-(-1)^{|f|}f\circ d$, and the {\em convolution} Lie bracket given by
$$
[f,g]=[\,,\,]\circ (f\otimes g)\circ\Delta,
$$
that is,
$$[f,g](c) =\sum_i (-1)^{\vert g\vert \, \vert c_i\vert} [f(c_i), g(c_i')] ,\quad\text{with}\quad \Delta c = \sum_i c_i\otimes c_i'\,.$$

\begin{rem}\label{rem1} If  $L$ is a cdgl then   $\Hom(C,L)$ is also complete. Indeed, let $\{F^n\}_{n\ge 1}$ be a filtration for which $L\cong \varprojlim_n L/F^n$ and consider the filtration $\{G^n\}_{n\ge 1}$ of $\Hom(C,L)$ given by $G^n=\Hom(C,F^n)$. Then,
$$
\Hom(C,L)\cong \varprojlim_n\Hom(C,L/F^n)\cong \varprojlim_n\Hom(C,L)/G^n.
$$
The same applies to note that $\Hom(\overline C,L)$ is also a sub cdgl of $\Hom(C,L)$ and there is a cdgl isomorphism,
\begin{equation}\label{iso}
\Hom(C,L)\cong \Hom(\overline C,L)\,\timest\, L,
\end{equation}
where both $L$ and  $\Hom(\overline C,L)$ are sub cdgl's and $[x,f]=\ad_x\circ f$ for  $x\in L$ and $f\in \Hom(\overline C,L)$. In particular, this twisted product is naturally a complete dgl.

Finally observe that the homotopy type of the cdgl $\Hom(C,L)$ is an invariant of the homotopy type of $C$. Indeed, if $\varphi\colon C\stackrel{\simeq}{\longrightarrow} C'$ is a cdgc quasi-isomorphism, the cdgl quasi-isomorphism $\varphi_*\colon \Hom(C',L)\stackrel{\simeq}{\longrightarrow} \Hom(C,L)$ trivially induces quasi-isomorphisms $\Hom(C',L)/{G'}^n\stackrel{\simeq}{\longrightarrow} \Hom(C,L)/{G}^n$. Thus, by the {\em generalized Goldman-Millson Theorem}, see \S\ref{prelimi}, $\varphi_*$ is a weak equivalence of cdgl's.

\end{rem}

For the remaining of this section we fix a connected cdgl $L$ and a connected simplicial set $X$  of finite type whose minimal model is denoted by $L'$. Recall  from \S\ref{prelimi} the notation $\{{L'}^n\}_{\ge 1}$ for the  filtration associated to  $L'$ and that $\quic$ denotes the chain coalgebra functor defined in (\ref{quillenlc}). We begin with a reformulation of \cite[Proposition 12.25]{bfmt0}:

\begin{theorem}\label{modelohom} The  realization of the  short exact cdgl sequence
$$
0\to\varinjlim_n\Hom(\overline\quic(L'/{L'}^n),L)\longrightarrow \varinjlim_n\Hom(\quic(L'/{L'}^n),L)\stackrel{\ev}{\longrightarrow} L\to 0
$$
has the homotopy type of the fibration sequence
$$
\Map^*(X,\langle L\rangle)\longrightarrow \Map(X,\langle L\rangle)\stackrel{\ev}{\longrightarrow} \langle L\rangle.
$$
In other words, there is a commutative diagram
$$
\xymatrix{
\Map^*(X, \langle L\rangle) \ar[r] & \Map(X, \langle L\rangle) \ar[r]^(.60){\ev} & \langle L\rangle\\
	\langle \varinjlim_n\Hom(\overline\quic(L'/{L'}^n),L)\rangle\ar[u]^\simeq \ar[r] & \langle\varinjlim_n\Hom(\quic(L'/{L'}^n),L) \rangle\ar[u]^\simeq \ar[r] &\langle L\rangle,\ar[u]^\simeq}$$
where the vertical maps are homotopy equivalences.
\end{theorem}
Here, the direct limit is taken in  $\catcdgl$ and $\varinjlim_n\Hom(\quic(L'/{L'}^n),L)\stackrel{\ev}{\to} L $ is induced by evaluating at $1\in \bq$ every morphism  of each $\Hom(\quic(L'/{L'}^n),L)$.
\begin{proof} Recall from \cite[Theorem 12.18]{bfmt0} or \cite[Theorem 6.6]{Berglund}  that for any cdga model A of any simplicial set $X$ and any cdgl $L$ there is a weak homotopy equivalence
$$
\langle A\otimesc L\rangle\simeq \Map(X,\langle L\rangle)
$$
natural on any possible choice. Moreover, see \cite[Proposition 12.25]{bfmt0}, there is a commutative diagram of simplicial sets
$$
\xymatrix{
\Map^*(X, \langle L\rangle) \ar[r] & \Map(X, \langle L\rangle) \ar[r]^(.60){\ev} &  \langle L\rangle\\
	\langle A^+ \widehat{\otimes} L\rangle\ar[u]^\simeq \ar[r] & \langle A\widehat{\otimes} L\rangle \ar[u]^\simeq \ar[r] &\langle L\rangle,\ar[u]^\simeq}
$$
where the vertical maps are homotopy equivalences and the bottom row is just the realization of the cdgl short exact sequence $A^+\otimesc L\to A\otimesc L\to L$.

On the other hand, if $A$ is a connected cdga of finite type, the cdgl  $A\otimesc L$ (respec. $A^+\otimesc L$) is naturally isomorphic to $\Hom(A^\sharp, L)$ (respec. $\Hom(\overline{A^\sharp}, L)$). Indeed, the following  chain of linear isomorphisms,
$$
 A\otimesc L=\varprojlim_n( A\otimes L/F^n)\cong \varprojlim_n \Hom(A^\sharp,L/F^n)\cong \Hom(A^\sharp, \varprojlim_n  L/F^n)\cong \Hom(A^\sharp,L),
 $$
 preserves differentials and Lie brackets.

Now, we choose
$$
A=\varinjlim_n {\mathscr C}^*(L'/{L'}^n)
$$
 which, by \cite[Theorem 10.8]{bfmt0}, is a Sullivan model of $X$. Observe that since each $L'/{L'}^n$ is of finite type, so is ${\mathscr C}^*(L'/{L'}^n)$. Then, as complete tensor products commute with inductive limits, the theorem follows from the isomorphism
$$
A\otimesc L\cong \varinjlim_n \quic^*(L'/{L'}^n)\otimesc L\cong\varinjlim_n\Hom(\quic(L'/{L'}^n),L)
$$
and its restriction to
$$
A^+\otimesc L\cong \varinjlim_n \quic^+(L'/{L'}^n)\otimesc L\cong\varinjlim_n\Hom(\overline\quic(L'/L'^n),L)
$$
where $\quic^+$ denotes the augmentation ideal of $\quic^*$.

\end{proof}

If $X$ is a nilpotent simplicial set of finite type the above result is drastically simplified. For it, the following observation is important.

\begin{rem}\label{nilpotente} Although outside of the classical Quillen theory for simply connected spaces, Neisendorfer defined in \cite[\S7]{nei} a ``Lie model'' of a nilpotent, finite type complex $X$ as any dgl (not complete) quasi-isomorphic to $\quil(A^\sharp)$. Here, $\quil$ is the functor in defined in (\ref{quillenlc}) and, as in the proof of Theorem \ref{modelohom}, $A$ stands for a finite type Sullivan model of $X$. On the other hand, if  $L'$ is the minimal (cdgl) model of $X$, then $L'$ and $\quil(A^\sharp)$ are quasi-isomorphic, see \cite[Theorem 10.2]{bfmt0}. Moreover, as stated in \S\ref{prelimi}, $\langle L'\rangle\simeq X_\bq$ which is also of the homotopy type of the cdga realization of $\quic^*\quil(A^\sharp)$. For a compendium of the geometrical realization in the commutative side see for instance \cite[\S1.6]{FHTII}.

\end{rem}

In what follows $
\Hom(\overline\quic(L'),L)\timest L$ will always denote the twisted product arising from the isomorphism (\ref{iso}) in Remark \ref{rem1}. The following extends \cite[Corollary 15]{bfm2}  the following result we consider, as in (\ref{iso}), the twisted product
$
\Hom(\overline\quic(L'),L)\timest L$.
\begin{corollary}\label{cor1}  Let $L'$ be a Lie model of the nilpotent simplicial set $X$ of finite type. Then, the  realization of the  short exact cdgl sequence
$$
0\to\Hom(\overline\quic(L'),L)\longrightarrow \Hom(\overline\quic(L'),L)\timest L{\longrightarrow} L\to 0
$$
has the homotopy type of the fibration sequence
$$
\Map^*(X,\langle L\rangle)\longrightarrow \Map(X,\langle L\rangle)\stackrel{\ev}{\longrightarrow} \langle L\rangle.
$$
\end{corollary}

\begin{proof}
As before, this fibration sequence has the homotopy type of the realization of the cdgl sequence $A^+\otimesc L\to A\otimesc L\to L$ where $A$ is the Sullivan minimal model of $X$. As $A$ is of finite type this sequence becomes
$$
0\to\Hom(A_+^\sharp,L)\longrightarrow \Hom(A^\sharp_+,L)\timest L{\longrightarrow} L\to 0.
$$
Now, see the properties of the Quillen pair (\ref{quillenlc}) in \S\ref{prelimi}, since $A^\sharp$ is a connected cdgc, it is quasi-isomorphic to $\quic\quil(A^\sharp)$. Observe, that $\quil(A^\sharp)$ is the {\em Neisendorfer model of $X$} which by Remark \ref{nilpotente}, is quasi-isomorphic to any given cdgl model of X. Hence, since $\quic$ preserves quasi-isomorphisms,
$$
A^\sharp\simeq \quic\quil(A^\sharp)\simeq\quic(L').
$$
This implies, in view of   the last observation in Remark \ref{rem1}, that the above cdgl sequence is homotopy equivalent to
$$
0\to\Hom(\overline\quic(L'),L)\longrightarrow \Hom(\overline\quic(L'),L)\timest L{\longrightarrow} L\to 0
$$
and the corollary follows.
\end{proof}
We now consider the restriction of the evaluation fibration to a given path component of $\map(X,\langle L\rangle)$,
$$
\map_f^*(X,\langle L\rangle)\longrightarrow\map_f(X,\langle L\rangle)\stackrel{\ev}{\longrightarrow} \langle L\rangle,
$$
 determined by a map $f\colon X\to\langle L\rangle$. Note that the fibre   is the non-connected complex of pointed homotopy classes of pointed maps freely homotopic to $f$.

\begin{rem}\label{remark} Given a Lie model $L'$ of a simplicial set $X$, the set  $\pi_0 \Map(X, \langle L\rangle)$ of path components of $\Map(X, \langle L\rangle)$ is the set of free homotopy classes $\{X,\langle L\rangle\}$ which, by Corollary \ref{coropointed}, is in bijective correspondence with the set $\{L',L\}/H_0(L)$. In the same way, by Corollary \ref{corouno}, the set $\pi_0\Map^*(X, \langle L\rangle)$ of pointed homotopy classes $\{X,\langle L\rangle\}^*$ is  identified with $\{L',L\}$.

Now, if $X$ is nilpotent, by Corollary \ref{cor1},
$$
\langle \Hom(\quic(L'),L)\rangle\simeq \Map(X,\langle L\rangle)\esp{and}\langle \Hom(\overline\quic(L'),L)\rangle\simeq \Map^*(X,\langle L\rangle).
$$
As a result,
$$
\pi_0\Map(X, \langle L\rangle)\cong \{L',L\}/H_0(L)\cong \widetilde\mc\bigl(\Hom(\quic(L'),L)\bigr)
$$
while, on the other hand,
$$
\pi_0\Map^*(X, \langle L\rangle)\cong\{L',L\}\cong \widetilde\mc\bigl(\Hom(\overline\quic(L'),L)\bigr).
$$
These correspondences can be explicitly described:

Let  $q\colon\overline\quic(L')\to L'$ be the degree $-1$ linear map defined by $q(sx)=x$ if $x\in L'$ and $q(\land^{\ge 2}sL')=0$. Denote in the same way the only possible extension of $q$ to $\quic(L')$ by $q(1)=0$. Now choose a map $f\colon X\to\langle L\rangle$ which can be assumed to be pointed since $L$ is connected. Then, the pointed homotopy class $[f]^*$ is identified to the homotopy class of a cdgl morphism $\varphi\colon L'\to L$ which corresponds to the gauge class of the MC element
\begin{equation}\label{varphipointed}
\overline\varphi=\varphi\circ q \colon\overline\quic(L')\longrightarrow L,
\end{equation}

 On the other hand, the free homotopy class $[f]$ corresponds to the class of $\varphi$ in $\{L',L\}/H_0(L)$ which is identified to the gauge class of the  MC element
\begin{equation}\label{varphifree}
\overline\varphi=\varphi\circ q \colon \quic(L')\longrightarrow L,
\end{equation}
\end{rem}

With this notation, consider the component
$
\Hom(\quic(L'),L)^{\overline\varphi}\to L
$
of $\Hom(\quic(L'),L)\to L$ at the MC element $\overline\varphi$. Recall from \S\ref{prelimi} that
$
\Hom(\quic(L'),L)^{\overline\varphi}
$
 is, by definition, the connected cover of $(\Hom(\quic(L'),L),D_{\overline\varphi})$ where $D_{\overline\varphi}$ is the differential $D$ in $\Hom(\quic(L'),L)$ perturbed by $\overline\varphi$. Then we have:

\begin{proposition}\label{escapa} The realization of the cdgl morphism
$$
\Hom(\quic(L'),L)^{\overline\varphi}\longrightarrow L
$$
has the homotopy type of the fibration
$$
\map_f(X,\langle L\rangle)\stackrel{\ev}{\longrightarrow} \langle L\rangle.
$$
\end{proposition}
\begin{proof} Recall that $\langle\Hom(\quic(L'),L)^{\overline\varphi}\rangle\simeq\langle\Hom(\quic(L'),L)\rangle^{\overline\varphi} $. To finish, note that in view of Corollary \ref{cor1} and Remark \ref{remark}, $\langle\Hom(\quic(L'),L)\rangle^{\overline\varphi}\to\langle L\rangle$ is of the homotopy type of $\map_f(X,\langle L\rangle)\stackrel{\ev}{\longrightarrow} \langle L\rangle$.
\end{proof}

 \begin{rem}\label{remark3} (1) Observe that $\Hom(\quic(L'),L)^{\overline\varphi}\to L$ is not in general a surjective morphism, i.e., it is not a cdgl fibration. Indeed, see (\ref{iso}), in $$(\Hom(\quic(L'),L),D_{\overline\varphi})\cong (\Hom(\overline\quic(L'),L)\timest L, D_{\overline\varphi}),$$
$D_{\overline\varphi}(x)=dx+[\overline\varphi,x]=dx-(-1)^{|x|}\ad_x\circ\overline\varphi$ for any $x\in L$. Therefore, elements of $L_0$ are not, in general, $D_{\overline\varphi}$-cycles and thus, they do not belong to $\Hom(\quic(L'),L)^{\overline\varphi}$.

(2) Note also that, in view of  Corollary \ref{cor1} and Remark \ref{remark}, the fibre of the map $\langle\Hom(\quic(L'),L)\rangle^{\overline\varphi}\to\langle L\rangle$ is the non-connected complex
$$
{\amalg}_{[\psi]}\,\langle  \Hom(\overline\quic(L'),L)\rangle^{\overline\psi}
$$
where: $[\psi]$ runs through homotopy classes of morphisms $\psi\colon L'\to L$  such that, when considered $\overline\psi$ in $\Hom(\quic(L),L)$, this is a $\mc$ element gauge related to $\overline\varphi$. A short computation shows that this amounts to the existence of $x\in L_0$ such that $e^{\ad_x}\circ\varphi=\psi$. In other terms, and with the notation in Definition \ref{granaccion}, $[\psi]=[x]\bulletchi[\varphi]$ with $[x]\in H_0(L)$. Each of these components is of the homotopy type of the corresponding component of $\map^*_f(X,\langle L\rangle)$.
\end{rem}

\begin{rem}\label{remark5} In all of the above, if $L$ is a Lie model of a connected simplicial set $Y$ of finite type, then the fibrations realized in Theorems \ref{modelohom}, Corollary \ref{cor1} and Proposition \ref{escapa} are, respectively, the fibration
$$
\Map^*(X,\bq_\infty Y)\longrightarrow \Map(X,\bq_\infty Y)\stackrel{\ev}{\longrightarrow} \bq_\infty Y
$$
and each of its components.
\end{rem}

\section{Derivations models of evaluation fibrations}\label{deriva}

The  results in the past section can be expressed in terms of derivations. To avoid excessive technicalities we assume $X$ nilpotent in what follows. Again $L'$ denotes a Lie model of $X$ and $L$ is a connected cdgl. Given a cdgl morphism $\varphi\colon L'\to L$,
consider the dgl morphism,
$$
\widetilde\varphi=\varphi\alpha\colon \quil\quic(L')\longrightarrow L,
$$
with $\alpha$  the adjunction map, see \S\ref{prelimi}. Next, endow
$$
\des\derr_{\widetilde\varphi}(\quil\quic(L'),L)
$$
with the desuspended differential,
$$
D(\des\theta)=-\des [d,\theta]=-\des(d\circ\theta-(-1)^{|\theta|}\theta\circ d),
$$
and
 a Lie bracket,
$$
[s^{-1}\gamma,\des\eta]=\des\theta,
$$
where $\theta$ is first defined on $\des\overline{\quic}(L')$ by
$$
-[\,\,,\,]\circ(\gamma\des\otimes\eta\des)\circ\overline\Delta\circ s.
$$
and then extended to $\quil\quic(L')$ as a $\widetilde\varphi$-derivation.

Consider  the twisted product
$$
(\des\derr_{\widetilde\varphi}(\quil\quic(L'),L)\timest L,D)
$$
where  $\des\derr_{\widetilde\varphi}(\quil\quic(L'),L)$ is a sub dgl and
$$[x,\des\theta]=(-1)^{|x|}\des(\ad_x\circ \,\theta),\qquad Dx=d x- \des(\ad_x\circ\,\overline\varphi\circ s), \quad x\in L,
$$
Here, $\ad_x\circ\,\theta$ and  $\ad_x\circ\, {\overline\varphi}\circ s$ denote the  $\widetilde\varphi$-derivations which are these compositions on $\des\overline\quic(L')$, and $d$ is the differential on $L$.
Then, Corollary \ref{cor1} translates to the following which extends \cite[Theorem 3]{bfm}:

\begin{theorem}\label{mappingder} For any cdgl morphism $\varphi\colon L'\to L$,
\begin{equation}\label{isoderi}
0\to\des\derr_{\widetilde\varphi}(\quil\quic(L'),L)\longrightarrow \des\derr_{\widetilde\varphi}(\quil\quic(L'),L)\,\timest\, L\longrightarrow L\to 0
\end{equation}
is a cdgl short exact sequence whose realization
has the homotopy type of the fibration sequence
\begin{equation}\label{fibratop}
\Map^*(X,\langle L\rangle)\longrightarrow \Map(X,\langle L\rangle)\stackrel{\ev}{\longrightarrow} \langle L\rangle.
\end{equation}
\end{theorem}

\begin{proof}
As perturbing a cdgl does not interfere with its realization, see (\ref{realizaigual}), Corollary \ref{cor1} amounts to say that (\ref{fibratop}) is homotopy equivalent to the realization of the cdgl sequence,
$$
0\to(\Hom(\overline\quic(L'),L),D_{\overline\varphi})\longrightarrow (\Hom(\overline\quic(L'),L)\timest L,D_{\overline\varphi})\stackrel{\ev}{\longrightarrow} L\to 0.
$$
 Recall that $D_{\overline\varphi}$ stands for the perturbed differential by the MC elements $\overline\varphi$ as in (\ref{varphipointed}) and (\ref{varphifree}) respectively. Next, as in \cite[Theorem 3]{bfm} we show that
\begin{equation}\label{bueniso}
      \Gamma\colon \des\derr_{\widetilde\varphi}(\quil\quic(L'),L)
      \stackrel{\cong}{\longrightarrow}
      ( \Hom(\overline\quic(L'),L),D_{\overline\varphi}),\quad\Gamma(s^{-1}\theta)(c)=(-1)^{|\theta|}\theta(s^{-1}c),
\end{equation}
   is a dgl isomorphism, which then exhibit  $\des\derr_{\widetilde\varphi}(\quil\quic(L'),L)$ as a complete dgl.

   It is trivially a linear isomorphism and a straightforward computation shows it commutes with the Lie brackets. It remains to prove that $\Gamma$ commutes with differentials. For it, let $\theta\in \derr_{\widetilde\varphi}(\quil\quic(L'),L)$ and $c\in \overline\quic(L')$ with $\overline\Delta(c)=\sum_ic_i\otimes c'_i$.

On the one hand,
$$
\begin{aligned}
\bigl(D_\varphi\Gamma(s^{-1}\theta)\bigr)(c)=&d\bigl(\Gamma(s^{-1}\theta)(c)\bigr)-(-1)^{|\theta|+1}
\Gamma (s^{-1}\theta)(dc)+[\overline\varphi,\Gamma (s^{-1}\theta)](c)\\
=&(-1)^{|\theta|}d\theta(s^{-1}c)+\theta(s^{-1}dc) + \sum_i (-1)^{|c_i|(|\theta|+1)+|\theta|}\sum_i[\overline\varphi(c_i),\theta(s^{-1}c'_i)].\\
\end{aligned}
$$

On the other hand,
$$
\begin{aligned}
\Gamma
\bigl(D(s^{-1}\theta)\bigr)(c)=&(-1)^{|\theta|}(d\circ\theta-(-1)^{|\theta|}\theta\circ d)(\des c)\\
=&(-1)^{|\theta|}d\theta(s^{-1}c)-\theta(d_1\des c)-\theta(d_2\des c)\\
=&(-1)^{|\theta|}d\theta(s^{-1}c)+\theta(\des dc)-\theta\bigl(\frac{1}{2}\sum_i
(-1)^{|c_i|}[s^{-1}c_i,s^{-1}c'_i]\bigr)\\
=&(-1)^{|\theta|}d\theta(s^{-1}c)+\theta(\des dc)+\sum_i(-1)^{|c_i|+|\theta|(|c_i|+1)}[\overline\varphi(c_i),\theta(s^{-1}c'_i)],
\end{aligned}
$$
and both expressions coincide.
In the last equality we use that $\quic(L')$ is cocommutative, $\theta$ is a $\widetilde\varphi$-derivation and $\widetilde\varphi(\des c)=-\overline\varphi(c)$ for any $c\in\quic(L')$.

To finish, note that the isomorphism $\Gamma$ trivially extends to
\begin{equation}\label{isoimp}
(\Hom(\overline\quic(L'),L)\timest L,D_{\overline\varphi})\cong(\des\derr_{\widetilde\varphi}(\quil\quic(L'),L)\,\timest\, L,D).
\end{equation}
Note  that, in view of the isomorphisms (\ref{bueniso}) and (\ref{isoimp}), together with Remark \ref{rem1}, the objects in (\ref{isoderi}) are cdgl's and the maps are cdgl morphisms.
\end{proof}

As a consequence we can easily express each component of (\ref{fibratop}) in terms of derivations.
Again, let $\Map_f(X,\langle L\rangle)$ denote the component of $\Map(X,\langle L\rangle)$ containing the map $f\colon X\to\langle L\rangle$, which corresponds, via Remark \ref{remark}, to an element in $\{L',L\}/H_0(L)$ represented by the homotopy class of a cdgl morphism $\varphi\colon L'\to L$.

\begin{corollary}\label{cor4}
The  realization of the  cdgl morphism
$$(\des\derr_{\widetilde\varphi}(\quil\quic(L'),L)\,\timest\, L,D)^0\longrightarrow L
$$
has the homotopy type of the fibration
$$
\Map_f(X,\langle L\rangle)\stackrel{\ev}{\longrightarrow} \langle L\rangle.
$$
\end{corollary}

Recall from \S\ref{prelimi} that given $M$ any cdgl, $M^0$ stands for its connected cover, or equivalently, its component at the MC element $0$.
\begin{proof} This is an immediate application of
 Proposition \ref{escapa} and the  isomorphism   (\ref{isoimp}).
\end{proof}
We finish by describing, also in terms of derivations, the homotopy long exact sequence of the evaluation fibration. For it, fix a pointed map $f\colon X\to\langle L\rangle$ as the base point of both, $\map_f^*(X,\langle L \rangle)$ and $\map_f(X,\langle L\rangle)$

On the other hand, let $L'$ be a Lie model of $X$ and let $\varphi\colon L'\to L$ be a cdgl morphism whose homotopy class in $\{L',L\}$ identifies the pointed homotopy class $[f]^*$. Recall again that, by Remark \ref{remark},  this morphism correspond to the free homotopy class $[f]$  in the orbit set $\{L',L\}/H_0(L)$. Consider the twisted chain complex
$$
(\derr_\varphi(L',L)\widetilde\times sL,D)
$$
which has $\derr_\varphi(L',L)$ as subcomplex and
$$
Dsx=- sdx+  ad_x\circ\varphi,\quad x\in L.
$$
Then, we prove the following which, in particular, recovers \cite[Theorem 12.35]{bfmt0}:
\begin{theorem}\label{corhom} For the chosen basepoints,
the homotopy long exact sequence of the fibration
\begin{equation}\label{puf}
\Map^*_f(X,\langle L\rangle)\longrightarrow \Map_f(X,\langle L\rangle)\stackrel{\ev}{\longrightarrow} \langle L\rangle
\end{equation}
is isomorphic to the homology long exact sequence of
$$
0\to\derr_{\varphi}(L',L)^{0}\longrightarrow (\derr_{\varphi}(L',L)\,\timest\, sL)^0\longrightarrow sL\to 0.
$$
In particular,
$$
\pi_*(\Map_f^*(X,\langle L\rangle),[f]^*)\cong H_*\bigl(\derr_\varphi(L',L)\bigr)\,\,\text{and}\,\, \pi_*\Map_f(X,\langle L\rangle)\cong H_*(\derr_\varphi(L',L)\timest sL).
$$
\end{theorem}
\begin{proof} On the one hand, since the adjunction map $\alpha\colon\quil\quic(L')\stackrel{\simeq}{\to}L'$ is a quasi-isomorphism and $L'$ is a Lie model (and thus cofibrant), the map of chain complexes,
\begin{equation}\label{alfa}
\alpha^*\colon \derr_\varphi(L',L)\stackrel{\simeq}{\longrightarrow}\derr_{\widetilde\varphi}(\quil\quic(L'),L),
\end{equation}
is also a quasi-isomorphism, see for instance \cite[Lemma 6]{bfm}.
Extend it  to
 another quasi-isomorphism
\begin{equation}\label{quasi}
\alpha^*\times\id_{sL}\colon (\derr_\varphi(L',L)\widetilde\times sL,D)\stackrel{\simeq}{\longrightarrow}(\derr_{\widetilde\varphi}(\quil\quic(L'),L)\widetilde\times sL,D')
\end{equation}
by setting
$$
D'sx=-sdx+\theta,\quad x\in L,
$$
where $\theta$ is the $\widetilde\varphi$-derivation which extends the map $\ad_x\circ\overline\varphi\circ s\colon s^{-1}\quic(L')\to L$. In particular,  $\theta$
 coincides with $\ad_x\circ\varphi$ on $L'\cong s^{-1}s L'$. This shows that (\ref{quasi}) is indeed a chain map and thus, a quasi-isomorphism. From (\ref{alfa}) and (\ref{quasi}) we obtain a commutative diagram of chain complexes  where the top sequence is the suspension of (\ref{isoderi}):
$$
\xymatrix{
0\ar[r]&\derr_{\widetilde\varphi}(\quil\quic(L'),L)\ar[r]&
(\derr_{\widetilde\varphi}(\quil\quic(L'),L)\widetilde\times sL,D')\ar[r]&sL\ar[r]&0\\
0\ar[r]&\derr_\varphi(L',L)\ar[u]^\simeq\ar[r]&
 (\derr_\varphi(L',L)\widetilde\times sL,D)\ar[u]^\simeq\ar[r]&
sL\ar[u]\ar[r]&0\\}
$$
Thus, by Corollary \ref{cor4}, the homology long exact sequence of the component at $0$ of the top row is precisely the homotopy long exact sequence of  (\ref{puf}). For it note that, by Remark \ref{remark3}(2), the component of $\map_f^*(X,\langle L\rangle)$ containing $[f]^*$ has the homotopy type of $\langle \Hom(\overline\quic(L'),L)^{\overline\varphi}\rangle$ and $ (\Hom(\overline\quic(L'),L),D_{\overline\varphi})\cong s^{-1}\derr_{\widetilde\varphi}(\quil\quic(L'),L)$.

\end{proof}

\section{Complete Lie algebras of derivations and twisted products}\label{comple}

Given $L$ a dgl (non necessarily complete in principle), the following twisted products  are key tools for our main results:

Given a cdgc $C$ consider the twisted product
\begin{equation}\label{clave}
\Hom(C,L)\longrightarrow \Hom(C,L)\,\timest\,\derr L\longrightarrow\derr L
\end{equation}
in which both terms are sub dgl's and
$$
[\theta,f]=\theta\circ f,\quad\theta\in\derr L,\quad f\in \Hom(C,L).
$$
This  dgl sequence, and its ``dual'' $\coder(C)\timest \Hom(C,L)$, have already proven to be useful in the modeling of certain simply connected spaces of homotopy automorphisms, see \cite[\S3]{ber2} and \cite[\S4]{ber3}.

Also,  consider the twisted product
\begin{equation}\label{slslrev1}
\derr L\longrightarrow \derr L \timest sL\longrightarrow  sL
\end{equation}
 where $sL$ is an abelian Lie algebra and
$$
Dsx=-sdx+\ad_x,\quad[\theta,sx]=(-1)^{|\theta|}s\theta(x),\quad x\in L,\quad \theta\in \derr L.
$$
Finally, let
\begin{equation}\label{slslrev2}
L\longrightarrow L\timest\derr L\longrightarrow \derr L
\end{equation}
be the twisted product in which  both terms are sub dgl's and $[\theta,x]= \theta(x)$ for any $\theta\in\derr L$ and $x\in L$.

If $L$ is a cdgl we observed in Remark \ref{rem1} that $\Hom(C,L)$  and any of its isomorphic forms, see for instance (\ref{iso}) or (\ref{isoimp}), are complete dgl's. However, $\derr L$ is not complete in general and, even if it is, $\Hom(C,L)\timest \derr L$ may fail to be so. The same applies to the twisted products $\derr L \timest sL$ and $L\timest\derr L$ just defined.

The following illustrates this situation.

\begin{ex} \label{example1} (1) Consider the cdgl $L=(\libc(x,y),0)$ where $|x|=|y|=0$. Observe that $\derr L=\derr_0 L$ and define $\theta,\gamma\in\derr(L)$  by
$$
\theta(x)=y,\quad \theta(y)=0,\qquad\gamma(x)=-\frac{1}{2}x, \quad\gamma(y)=\frac{1}{2}y.
$$
A short computation shows that $[\gamma,\theta]=\theta$ and thus $\ad^n_\gamma(\theta)=\theta$ for any $n\ge 1$. In particular $\theta$ lives in the kernel of  the canonical map $\derr L\to \varprojlim_{n\ge 1 }\derr L/(\derr L)^n$ which, therefore, prevents $\derr L$ to be complete.

(2) On the other hand, consider an odd dimensional sphere $S^n$ whose model is the abelian Lie algebra $L=(\lib(x),0)$ with $x$ of degree $n-1$. In this case $\quic(L)=(\land sx,0)$. Then, $\Map(S^n_\bq,S^n_\bq)$ is modeled by $\Hom(\quic(L),L)$ which is an abelian Lie algebra generated by $x$, corresponding to $1\mapsto x$, and an element $z$ of degree $-1$, corresponding to $sx\mapsto x$. On the other hand $\derr L$ is  an abelian lie algebra generated by $\id_L$, a degree $0$ derivation which we denote by $\theta$. Hence, in this case,  (\ref{clave}) takes the form
$$
\Span\{x,z\}\longrightarrow \Span\{x,z,\theta\}\longrightarrow \Span\{\theta\}
$$
where, in the middle, $[\theta,x]=x$ and  $[\theta,z]=z$. Observe then that the bracket does not respect the usual filtration in the non twisted product. Note also that, even though $\Hom(\quic(L),L)$ and $\derr L$ are obviously complete as they are abelian, $\Hom(\quic(L),L)\,\timest\, \derr L$ is not (with respect to the usual bracket filtration) as its completion is easily seen to be just  $\Span\{\theta\}$. That is, the completion of (\ref{clave}) becomes
$$
\Hom(\quic(L),L)\stackrel{0}{\longrightarrow} \derr L\stackrel{\id}{\longrightarrow} \derr L.
$$
 This  also illustrates that the completion functor is not left exact in general.

 (3) Finally, with $L$ as in (2) above, observe that the twisted dgl $\derr L\timest sL$ becomes the dgl
 $\Span\{\theta,sx\}$, with zero differential and
$[\theta,sx]=sx$, whose  completion  is  $\derr L=\Span\{\theta\}$. An analogous argument shows that $L\timest\derr L$ fails also to be complete with respect to the usual filtration.

\end{ex}

The purpose of this section is to overcome these obstacles by imposing some restrictions which are sufficiently mild for our goals.

\medskip

From now on we fix $L=(\libc(V),d)$  a connected, minimal, free cdgl and  a finite filtration of  $V$ by graded subspaces:
\begin{equation}\label{centralv}
 V=V^0\supset\dots\supset V^i\supset V^{i+1}\supset\dots\supset V^q=0,
\end{equation}

We first refine the usual filtration in $L$ by considering,
for $n\ge 1$ and $p\ge0$,
$$
\libc^{n,p}(V)=\Span\{\bigl[a_1,[a_2,[\dots,[a_{n-1},a_n]\bigr] \dots\bigr]\in\libc^n(V),\,a_i\in V^{\alpha_i}\,\text{and}\,\,{\textstyle \sum_{i=1}^n}\alpha_i=p\}.
$$
Observe that $\libc^{n,p}(V)=0$ for $p\ge nq$. Then, for $n\ge 1$ and $0\le p\le nq-1$, define
$$
F^{n,p}=\libc^{n,p}(V)\oplus \libc^{\geq n+1}(V)
$$
and note that,
$$
\begin{aligned}
\libc(V)&=F^{1,0}\supset F^{ 1,1}\supset\dots\supset F^{1,q-1} \\
&\supset F^{ 2,0}\supset F^{ 2,1} \supset\dots\supset F^{2,2q-1}\\
&...............................................\\
&\supset F^{n,0}\supset F^{n,1}\supset\dots\supset F^{n,nq-1}\\
&...............................................\\
\end{aligned}
$$
Remark that  $F^{ n,p}$
 ranks
\begin{equation}\label{monorden}
m=q+\dots+(n-1)q+p+1=\frac{(n-1)n q}{2}+p+1
\end{equation}
in the order given by this chain of inclusions.
\begin{definition}\label{filtranueva} For $n,p$ and $m$ as above define
$$
F^m=F^{n,p}.
$$
\end{definition}

\begin{proposition}\label{prop:filtrationL} The sequence $\{F^m\}_{m\geq 1}$ is a filtration of $L$. Furthermore,  $L$ is complete with respect to this filtration.
\end{proposition}
\begin{proof}
Observe that $[F^{n,p},F^{r,s}]\subset F^{n+r,p+s}$. Now, $F^{n,p}$ and $F^{r,s}$ rank respectively
$
\frac{(n-1)nq}{2}+p+1$ and $\frac{(r-1)rq}{2}+s+1
$
in the filtration order.  The sum of these integers is smaller than
$\frac{(n+r-1)(n+r)q}{2}+p+s+1$ which is the position of $F^{n+r,p+s}$ in that order. This shows that the Lie bracket is filtration preserving.

On the other hand, since the differential of $L$ is decomposable each $F^m$ is a differential ideal. Moreover, we have
$$dF^m\subset F^{m+1}$$

We now check that $L$ is complete respect to this filtration. Since $\cap_{n,p}F^{n,p}=0$  the natural map
$$
\libc(V) \longrightarrow\varprojlim_{m} \libc(V) / F^{m}
$$
is injective. Finally, write a given element in $\varprojlim_{n,p}\libc(V) /F^{n,p}$  as a series,
$$
\sum_{n,p}x_{n,p}\esp{with}x_{n,p}\in F^{n,p},
$$
and note that this is a well defined element in $\libc(V)$ since, for each $m\ge1$, $\sum_{n,p}x_{n,p}$ contains only a finite sum in $\libc^m(V)$.  This implies the surjectivity of the map above.
\end{proof}

Note that considering in $L$ the usual filtration $\{\libc^{\ge m}(V)\}_{m\ge 1}$ for which it is complete, the identity $(L,\{\libc^{\ge m}(V)\}_{m\ge 1})\to (L,\{F^m\}_{m\ge 1})$ is a cdgl morphism.

\begin{definition}\label{derk} Denote by $\der L\subset\derr L$  the connected sub dgl in which,
$$
\begin{aligned}
&\der_{\ge1} L= \derr_{\ge1}L,\\
&\der_0L=\{\theta\in\derr_0L \cap \ker D,\,\,\theta(V^i)\subset V^{i+1}\oplus \libc^{\ge 2}(V)\,\text{for all $i$}\,\}.
\end{aligned}
$$
Note that this is a well defined Lie algebra and, since the differential in $L$ is decomposable, it is a differential sub Lie algebra.
\end{definition}

Using the new filtration on $L$ we first construct a  filtration of $\der_0 L$ for which it becomes a cdgl. For it, note that
$$
\der_0 L=\{\theta\in\derr_0L \cap \ker D,\,\,\theta(F^m)\subset F^{m+1}\,\,\text{for all $m$}\,\}
$$
Hence, we filter $\der_0 L$ by,
\begin{equation}\label{filrev}
 \fil^n=\{\theta\in\der_0 L,\,\,\theta(F^m)\subset F^{m+n},\,\,\text{for all $m$}\,\},\quad n\ge 1,
\end{equation}
and a straightforward computation proves:

\begin{lemma}\label{filderk} $\der_0 L$ is complete with respect to $\{\fil^n\}_{n\ge 1}$.\hfill$\square$
\end{lemma}

To finish, we ``extend'' this filtration on $\der_0 L$ to $\der L$ by means of the following procedure:

Consider first the subspaces of $\derr_{\ge 1}L$,
$$
M^n=\{\theta\in\derr_{\ge 1} L,\,\theta(L)\subset L^{\ge n+1}\},\qquad n\ge 1,
$$
and choose a refinement of this sequence by constant subspaces
$$
\begin{aligned}
M^1=&\calj^{1,0}=\calj^{ 1,1}=\dots= \calj^{1,q-1}\supset\\
M^2=&\calj^{ 2,0}=\calj^{ 2,1} =\dots=\calj^{2,2q-1} \supset \\
&..................................................\\
M^n=&\calj^{ n,0}=\calj^{ n,1} =\dots=\calj^{n,nq-1} \supset \\
&..................................................\\
\end{aligned}
$$
As in (\ref{monorden}),   $\calj^{ n,p}$
 ranks
$m=\frac{(n-1)n q}{2}+p+1$ in the order given by this chain, and define
$$
\calj^m=\calj^{n,p}.
$$
The only purpose of extracting this refinement is simply to assure that,
$$
[\fil^n,\calj^m]\subset \calj^{m+n},\quad n,m\ge 1.
$$
Finally, for any $n\ge 1$ define the graded vector space $E^n$ as follows:
\begin{equation}\label{filder}
E^n_p=\begin{cases} \fil^n&\text{if $p=0$,}\\\calj^{n-p}_p&\text{if $1\le p\le n-1$,}\\\derr_p L&\text{if $p\ge n$.}
\end{cases}
\end{equation}

\begin{proposition}\label{derkacomp}   $\der L$ is a cdgl with respect to the filtration $\{E^n\}_{n\ge 1}$. \end{proposition}
\begin{proof}
Note that, $E^1=\fil^1\oplus\derr_{\ge 1}L=\der L$. Also, since the differential in $L$ is decomposable, $D\calj^m_p\subset \calj^{m+1}_{p-1}$ for   $p\ge 2$, and $D\calj^{m}_1\subset \fil^{m+1}$,  for all $m\ge 1$. Finally $[\calj^n,\calj^m]\subset\calj^{n+m}$ while, as noted above,  $[\fil^n,\calj^m]\subset \calj^{m+n}$ for $n,m\ge 1$. With this data one easily checks that $\{E^n\}_{\ge 1}$ is a filtration of $\der L$.

Since, for $p\ge 1$ and $n\ge 0$, we have
$$
E^n_p=\begin{cases} \derr_p L&\text{if $n\le p$,}\\\calj^{n-p}_p&\text{if $n>p$},
\end{cases}
$$
it follows that, for $p\ge 1$,
$$
\varprojlim_n(\der L/E^n)_p=\varprojlim_{n} \derr_p L/\calj^n_p=\derr_p L.
$$
On the other  hand, for $p=0$, and in view of Lemma \ref{filderk},
$$
\varprojlim_n(\der L/E^n)_0=\varprojlim_n\der_0 L/\fil^n=\der_0 L,
$$
\end{proof}

Next, given $f\colon (\libc(V),d)\stackrel{\cong}{\to} (\libc(V),d)$ an element of $\aut L$, denote by
$$
f_*\colon V\stackrel{\cong}{\longrightarrow} V
$$
the induced automorphism on the indecomposables. Then,  we prove:

\begin{proposition}\label{exponencial}
$
\exp\bigl(\der_0 L\bigr)=\{f\in\aut(L),\,\,(f_*-\id_V)(V^i)\subset V^{i+1}\,\,\text{for all $i$}\,\}.
$
\end{proposition}

\begin{proof} Let $ \theta\in\der_0 L$ and call $f=e^\theta$. Then $f-\id_L=\sum_{n\ge 1}\frac{\theta^n}{n!}$ which clearly satisfies $(f-\id_L)(V^i)\subset V^{i+1}\oplus \libc^{\ge 2}(V)$ for all $i$. Hence $(f_*-\id_V)(V^i)\subset V^{i+1}$ also for all $i$.

\smallskip

Conversely, let $f\in\aut(L)$ be such that $f(v)-v\in V^{i+1}\oplus\libc^{\ge 2}(V)$ for all $v\in V^i$ and all $i$. Define a linear morphism,
$$
\theta\colon V\longrightarrow L,\qquad \theta(v)=\sum_{n\geq 1} (-1)^n\frac{(f-\id_L)^n(v)}{n}
$$
Recursively, one sees that
the component of $(f-\id_L)^n(v)$ in $\libc^{m,r}(V)$ is zero for $n$ large enough. Thus, $\theta$ is  well defined and can be extended to a  derivation $\theta \in \der_0(L)$. Finally, since the formula that defines $\theta$ is that of  $\log f$, computing $e^\theta$ recovers $f$.
\end{proof}

Next, recall from  Remark \ref{rem1} that, for any cdgc $C$ and any cdgl $L$, $\Hom(C,L)$ is complete with respect to the natural filtration given by that of $L$ . However, we will need a new filtration which is ``compatible'' with the one we just constructed for $\der L$.

\begin{definition}\label{defifil}
Let $C$ be a cdgc and $L=(\libc(V),d)$ be a minimal cdgl filtered by $\{F^m\}_{m\geq 1}$ of Definition \ref{filtranueva}. For $m\geq 1$, define the non negatively graded vector space $I^m$ by,
$${I}^m_k= \Hom(C,F_k^{m-k}),\qquad\text{for $k\ge0$},$$
where  $F^{n}=L$ for $n\leq 0$.
\end{definition}

\begin{proposition}\label{homcomp}
The sequence $\{{I}^m\}_{m\geq 1}$ is a filtration of $\Hom(C,L)$ for which it becomes a cdgl.
\end{proposition}

\begin{proof} Clearly $I^1=\Hom(C,L)$, $I^{m+1}\subset I^m$  and $D{I}^m\subset {I}^m$ for all $m$.  On the other hand, given $\varphi\in \Hom(C,F^{n-i}_i)\subset I^{n}$, $\psi\in\Hom(C,F^{m-j}_j)\subset I^{m}$ and $c\in C$ with $\Delta(c)=\sum_i c_i\otimes c_i'$,
$$[\varphi,\psi](c)=\pm \sum_i[\varphi(c_i),\psi(c_i')]\in F^{n+m-i-j}_{i+j}.$$
This implies that $[I^{n},I^{m}]\subset I^{n+m}$ and therefore $\{{I}^m\}_{m\ge 1}$  is a filtration. A degree-wise analogous argument to that in Remark \ref{rem1} shows that $\Hom(C,L)$ is complete with respect to this new filtration.
\end{proof}
Next, in the restriction of the twisted product in (\ref{clave}) to
$$
\Hom(C,L)\,\timest\,\der L,
$$
consider $
\{J^n\}_{n\ge 1}=\{I^n\times E^n\}_{n\ge 1}
$
with  $\{I^n\}_{n\ge 1}$  of Definition \ref{defifil} and $\{E^n\}_{n\ge 1}$ as in (\ref{filder}).
Then, a direct inspection proves:

\begin{proposition}\label{twistcomp}
The sequence $\{{J}^n\}_{n\geq 1}$ is a filtration of $
\Hom(C,L)\,\timest\,\der L
$ for which it becomes a cdgl.\hfill$\square$
\end{proposition}

Furthermore, from Propositions \ref{derkacomp}, \ref{homcomp} and \ref{twistcomp}, and with respect to  the corresponding filtrations, we  deduce

\begin{corollary}\label{coroimp} The maps
$$
  \Hom(C,L)\longrightarrow \Hom(C,L)\,\timest\,\der L\longrightarrow\der L
  $$
form a short exact sequence of cdgl morphisms and thus, it is a cdgl fibration. \hfill$\square$
\end{corollary}

Finally, in the twisted products
$$
\der L \timest sL\esp{and}L\timest\der L,
$$
obtained as restrictions of those in (\ref{slslrev1}) and (\ref{slslrev2}), we consider the respective sequences
$$
\{E^n\times sF^n\}_{n\ge 1}\esp{and}\{F^n\times E^n\}_{n\ge 1}.
$$
By procedures analogous to those used in this section one obtains:

\begin{proposition}\label{excusa} These sequences are filtrations for which
$
\der L \timest sL$ and $L\timest\der L,
$
are cdgl's.\hfill$\square$
\end{proposition}

 \section{Nilpotent monoids of homotopy automorphisms and classifying fibrations}\label{princi}

Whenever $X$ is a simply connected finite complex with minimal model $L$,    it is well known (see \cite{SS1,tan}, \cite{ber2} in the fiberwise context or  \cite{ber3} for the relative case) that the  fibration,
$$
 X_\bq\longrightarrow B\aut^*_1(X_\bq)\longrightarrow  B\aut_1(X_\bq)
$$
is modeled by a dgl fibration sequence of the form
$$ L\stackrel{\ad}{\longrightarrow}\widetilde\derr L\longrightarrow \widetilde\derr L \timest sL.
$$
where $ \widetilde\derr L \timest sL$ is the restriction of the twisted product (\ref{slslrev1}).
Again for $X$ simply connected, the above fibration is the simply connected cover of
 the universal classifying fibration
$$
 X_\bq\longrightarrow B\aut^*(X_\bq)\longrightarrow  B\aut(X_\bq).
 $$
However, in general and  even if  $X$ is simply connected, this classifying fibration cannot be modeled. Indeed, as the following result shows, $B\aut^*(X_\bq)$ and $B\aut(X_\bq)$ do not lie, in general, in the image of the realization functor. In particular, they are not nilpotent spaces.

\begin{proposition}\label{propo20} For any $n\ge 1$, neither $B\aut^*(S^n_\bq)$ nor $B\aut(S^n_\bq)$  have the homotopy type of the realization  of any connected cdgl.
\end{proposition}

\begin{proof}
Suppose $B\aut^*(S^n_\bq)\simeq \langle L\rangle$ for a given connected cdgl $L$. Then there is a group isomorphism
$$
\pi_0\aut^*(S^n_\bq)\cong\pi_1B\aut^*(S^n_\bq)\cong H_0(L)
$$
where in the latter, the group structure is given by the BCH product. However, by Remark \ref{remark}, $\pi_0\map^*(S^n_\bq)=\{\lib(x),\lib(x)\}*=\{\lib(x),\lib(x)\}$. Hence, $\pi_0\aut^*(S^n_\bq)\cong \{\lib(x),\lib(x)\}-\{0\}$ are just the automorphisms of a one dimensional vector space which is identified with the multiplicative group $\bq^*=\bq-\{0\}$.

Now, assume there is an isomorphism $\psi\colon \bq^*\stackrel{\cong}{\to} H_0(L)$ and let $a=\psi(2)$. Recall that,  for the BCH product, $\mu a*\nu a=(\lambda+\mu)a$ for $\mu,\nu\in \bq$. Thus, given $\lambda\in\bq^*$ with $\psi(\lambda)=\frac{1}{2}a$ we have
$
\psi(\lambda^2)= \frac{1}{2}a*\frac{1}{2}a=a=\psi(2)$.  Then $\lambda^2=2$ which is a contradiction. The same argument also works for the non-pointed case.

\end{proof}
Despite this result we will be able to realize distinguished classifying fibrations, also by means of cdgl's of derivations.

Let $X$ be a connected finite complex and let $\pi_0\aut(X)$  be the group of homotopy classes of self homotopy equivalences of $X$,  often denoted  by ${\cal E}(X)$. In the same way, $\pi_0\aut^*(X)$ is the group ${\cal E}^*(X)$ of based homotopy classes of pointed self homotopy equivalences of $X$. Denote by
\begin{equation}\label{surfun}
\zeta\colon \cale^*(X)\longrightarrow \cale(X)
\end{equation}
 the natural surjection which induces the bijection $\cale^*(X)/\pi_1(X)\cong\cale(X)$.

\begin{definition}\label{gyg} For a given subgroup $G\subset \cale(X)$ we consider the sub monoids $\aut_G(X)\subset \aut(X)$ and  $\aut_G^*(X)\subset \aut^*(X)$ defined by
$$
\aut_G(X)=\{f\in\aut( X),\,\,[f]\in G\},\quad \aut_G^*(X)=\{f\in\aut^* (X),\,\,[f]\in G\}
$$
\end{definition}
Note that  $\pi_0\aut_G(X)=G$ while $\pi_0\aut_G^*(X)$ is the subgroup $G^*\subset\cale^*(X)$ consisting of pointed homotopy classes  of pointed homotopy automorphisms  in $\aut_G^*(X)$. In other terms,
$$
G^*=\zeta^{-1}(G)=\{[f]^*\in\cale^*(X),\,\,\text{with $[f]\in G$}\}.
$$
Thus, note that $G^*$ is preserved by the action of  $\pi_1(X)$ and
$$
 G^*/\pi_1(X)\cong G.
$$
Then, the evaluation fibration $\map^*(X,X)\longrightarrow\map(X,X)\longrightarrow X$ restricts to a fibration $\aut_G^*(X)\to \aut_G(X)\to X$ which is extended on the right to provide a fibration sequence
$$
X\longrightarrow B\aut^*_G(X)\longrightarrow B\aut_G(X).
$$
This fibration sequence, see  \cite{fuen}, is universal with respect to fibrations with fiber $X$ and whose image of the holonomy action is contained in $G$. In other terms, the set
\begin{equation}\label{clasefib}
\fib_G(X,B)
\end{equation}
 of equivalence classes of  fibration sequences $X\to E\to B$ over $B$ for which  the image of the natural map $\pi_1(B)\to\cale(X)$  lies in $G$, is in bijective correspondence with the set of (free) homotopy classes $\{B,B\aut_G(X)\}$.

\medskip

 The counterpart of the above in the pointed setting is the following:

 \begin{definition}\label{pyp} Given  $\pi\subset \cale^*(X)$  a subgroup of pointed homotopy classes of pointed equivalences, define
$$
\aut_\pi^*(X)=\{f\in\aut^*(X),\,\,[f]^*\in\pi\}.
$$
\end{definition}

\begin{rem}\label{explico} Note that in general, given $G\subset\cale(X)$ and $\pi\subset\cale^*(X)$, $\aut_G^*(X)$ and $\aut_{\pi}^*(X)$ are different submonoids of $\aut^*(X)$. In particular, for the trivial subgroup $G=\{1\}\subset \cale(X)$, $\aut_1^*(X)$ is the non connected monoid of pointed maps freely homotopic to $\id_X$. However, for $\pi=\{1\}\subset\cale^*(X)$, $\aut_1^*(X)$ is the connected monoid of pointed maps  based homotopic to $\id_X$. In what follows, and to avoid confusion, we always make clear whether the considered subgroup is taken from $\cale(X)$ or $\cale^*(X)$.
\end{rem}
Recall that a fibration of pointed spaces is {\em pointed} if it has a section. Then, see for instance \cite[Theorem 2.2]{fuen},  there exists a ``universal'' pointed fibration,
\begin{equation}\label{clasi}
X\longrightarrow Z\longrightarrow B\aut_\pi^*(X)
\end{equation}
such that:  the set
$$
\fib_\pi^*(X,B)
$$
 of equivalence classes of  pointed fibrations  $F\to E\to B$ over $B$, with fiber homotopically equivalent to $X$, and for which  the image of the natural map $\pi_1(B)\to\cale^*(F)$  lies in $\pi$, is in bijective correspondence with the set of free homotopy classes $\{B,B\aut_\pi^*(X)\}$.

 \medskip

From this point on, we fix a finite nilpotent complex $X$. We also recall that an action of a group $G$ on an abelian group $H$ is {\em nilpotent} if the lower central series of the action,
$$
\Gamma^0\supset\dots\supset\Gamma^{i-1}\supset\Gamma^i\supset\dots
$$
where $\Gamma^0=H$ and $\Gamma^i$ is generated by $\{gh-h,\,g\in G,\,h\in \Gamma^{i-1}\}$, is finite. That is, $\Gamma^q=0$ for some $q$.

\medskip

Let first $G$ be a subgroup of ${\cal E}(X)$ which acts nilpotently on $H_*(X)$. Then, see \cite[Theorem D]{drorza},  both $B\aut_G(X)$ and $B\aut_G^*(X)$ are nilpotent spaces and thus $G\cong\pi_1 B\aut_G(X)$ and $G^*\cong\pi_1 B\aut_G^*(X)$ are nilpotent groups. Let
$$
\Gamma^0=H_*(X)\supset\dots\supset\Gamma^{i-1}\supset\Gamma^i\supset\dots\supset\Gamma^q=0
$$
be the central series of the $G$-action. That is, $\Gamma^i$ is generated by $f_*(\alpha)-\alpha$, with $[f]\in G$ and $\alpha\in \Gamma^{i-1}$.

\begin{definition}\label{K} Define $K\subset \cale(X)$ as the subgroup of self homotopy equivalences which ``stabilize'' the above series. That is,
$$
K=\{[f]\in\cale(X),\,\,\text{$f_*$ induce the identity on $\Gamma^i/\Gamma^{i+1}$, for all $i$}\,\}.
$$
Accordingly, define
 $$
K^*=\zeta^{-1}(K)=\{[f]^*\in\cale^*(X),\,\,\text{with $[f]\in K$}\}
$$
and note that $G\subset K\subset\cale(X)$ and $G^*\subset K^*\subset\cale^*(X)$. Observe that, by \cite[Theorem D]{drorza}, both $K$ and $K^*$ are  nilpotent groups.

\end{definition}

On the one hand  the following  straightforward homological reformulation of \cite[Theorem 3.3]{ma} characterizes the rationalization of $K$:

\begin{theorem}\label{teoracio} The  rationalization $K_\bq$ of $K$  is the subgroup of  ${\cal E}(X_\bq)$ which stabilize the following series,
\begin{equation}\label{final}
\Gamma^0_\bq=H_*(X_\bq)=H_*(X)_\bq\supset\dots\supset\Gamma^{i-1}_\bq\supset\Gamma^i_\bq\supset
\dots\supset\Gamma^q_\bq=0.
\end{equation}
 In other terms, a class $[f]$ of ${\cal E}(X_\bq)$ is in $K_\bq$ if it acts trivially on each $\Gamma^{i-1}_\bq/\Gamma^i_\bq$ for all $i\ge 0$. Moreover, the map
$
 K\to K_\bq$, $[f]\mapsto[f_\bq]$,
 is the rationalization. \hfill$\square$
 \end{theorem}
 As $G$ is a subgroup of $K$, it follows that $G_\bq$  is a subgroup of $K_\bq$ and the map $
 G\to G_\bq$, $[f]\mapsto[f_\bq]$,
 is also the rationalization.

On the other hand, a completely analogous procedure  exhibits  $K^*_\bq$ containing $G^*_\bq$  and establishes that the map $
 G^*\to G^*_\bq$, $[f]^*\mapsto[f_\bq]^*$ is the rationalization.

 \begin{rem}\label{rev2} This defines  maps of monoids,
 $$
 \aut_G(X) \longrightarrow\aut_{G_\bq}(X_\bq),\qquad  \aut_G^*(X) \longrightarrow\aut_{G_\bq}^*(X_\bq),
 $$
 both sending a map to its rationalization. Moreover, and this is crucial in the next result, remark that, componentwise, these maps are also the rationalizations as, for each map $f\colon X\to X$, pointed in the second case,
 $$
 \map_{f_\bq}(X_\bq,X_\bq)\simeq\map_f(X,X)_\bq,\quad \map_{f_\bq}^*(X_\bq,X_\bq)\simeq\map_f^*(X,X)_\bq.
 $$
 \end{rem}
From all of the above, we may easily deduce:
\begin{proposition}\label{rationalbaut}
Let $X$ be a finite nilpotent complex and $G\subset\cale(X)$ be a subgroup which acts nilpotently on $H_*(X)$. Then, the rationalization of the classifying fibration sequence
\begin{equation}\label{rev3}
X\longrightarrow B\aut^*_G(X)\longrightarrow B\aut_G(X)
\end{equation}
has the homotopy type of the classifying fibration sequence
\begin{equation}\label{rev4}
X_\bq\longrightarrow B\aut^*_{G_\bq}(X_\bq)\longrightarrow B\aut_{G_\bq}(X_\bq).
\end{equation}
\end{proposition}
\begin{proof} Observe that the rationalization of (\ref{rev3})
$$
X_\bq\longrightarrow \bigl(B\aut^*_G(X)\bigr)_\bq\longrightarrow \bigl(B\aut_G(X)\bigr)_\bq
$$
is a fibration sequence which lies in $\fib_{G_\bq}\bigl(X_\bq,\bigl(B\aut_G(X)\bigr)_\bq\bigr)$ and thus, by the classifying property of (\ref{rev4}), it fits in a commutative diagram as follows:
\begin{equation}\label{rev5}
\xymatrix{X_\bq\ar[r]\ar@{=}[d]\ar[r]& \bigl(B\aut^*_G(X)\bigr)_\bq\ar[d]\ar[r]& \bigl(B\aut_G(X)\bigr)_\bq\ar[d]\\
X_\bq\ar[r]&B \aut_{G_\bq}^*(X_\bq)
 \ar[r]&
B \aut_{G_\bq}(X_\bq).\\}
\end{equation}
This can be extended to a homotopy commutative diagram of the form
$$
\xymatrix{\aut_G^*(X)_\bq\ar[r]\ar[d]_\simeq\ar[r]&\aut_G(X)_\bq\ar[d]_\simeq\ar[r]&
X_\bq\ar[r]\ar@{=}[d]\ar[r]& \bigl(B\aut^*_G(X)\bigr)_\bq\ar[d]\ar[r]& \bigl(B\aut_G(X)\bigr)_\bq\ar[d]\\
\aut_{G_\bq}^*(X_\bq)\ar[r]&\aut_{G_\bq}(X_\bq)
 \ar[r]&X_\bq\ar[r]&B \aut_{G_\bq}^*(X_\bq)
 \ar[r]&
B \aut_{G_\bq}(X_\bq).\\}
$$
where both rows are again fibration sequences and, in view of Remark \ref{rev2}, the first two vertical arrow are homotopy equivalences. Hence, all the vertical arrows in (\ref{rev5}) are also homotopy equivalences and the proposition follows.
\end{proof}

Now, let $L=(\libc(V),d)$ be the minimal Lie model of $X$. In view of Corollary \ref{corouno}, $\{X_\bq,X_\bq\}^*\cong\{L,L\}$. In particular, every pointed homotopy equivalence $X_\bq\stackrel{\simeq}{\to}X_\bq$ in $\cale^*(X_\bq)$ is identified with a homotopy class of a quasi-isomorphism $L\stackrel{\cong}{\to}L$ which is necessarily an automorphism as $L$ is connected \cite[Theorem 3.19]{bfmt0}.

\begin{definition}\label{explog} Denote by
$$
\calge\subset \aut(L)/\sim,
$$
the subgroup of homotopy classes of automorphisms of $L$ that corresponds to $G_\bq^*\subset\cale^*(X_\bq)$ under the isomorphism $\{L,L\}\cong\{X_\bq,X_\bq\}^*$.

As $G_\bq^*/\pi_1(X_\bq)\cong G_\bq$, by Corollary \ref{coropointed} we deduce  that $\calge$ is invariant by the $H_0(L)$ action and
$$
\calge/H_0(L)\cong G_\bq.
$$
Denote also,
 $$
 \aut_{\calge}(L)=\{\varphi\in\aut(L),\,[\varphi]\in \calge\},
 $$
 which is clearly a subgroup of $\aut(L)$. Accordingly, for $K\subset \cale(X)$ as in Definition \ref{K}, define
 $$
 \calk\subset\aut(L)/\sim
  $$
  isomorphic to $K_\bq^*\subset\cale^*(X_\bq)$ and
  $$
  \aut_\calk(L)=\{\varphi\in\aut(L),\,\,[\varphi]\in\calk\}.
  $$
 \end{definition}

 The counterpart of these groups are the following fundamental lie algebras of derivations:

 \begin{definition}\label{explogderK}
Define  $\derk L\subset \derr L$ as the connected  cdgl where
$$
\derk_{\ge 1}L= \derr_{\ge1}L\quad\text{and}\quad \derk_0L=\{\text{$\theta\in\derr_0L$ such that  $D\theta=0$  and $e^{\theta}\in\aut_{\calk}(L)$}\}.
$$
\end{definition}

\begin{rem}\label{derkcomplete}
Observe that $\derk L$ is in fact a  well defined cdgl: recall that  $L=(\libc(V),d)$ is the minimal model of $X$.  By the isomorphism $s^{-1}\widetilde{H}_*(X_\bq)\cong V$, the series    (\ref{final}) of Theorem \ref{teoracio}  is then of the form
  $$
 V=V^0\supset\dots\supset V^i\supset V^{i+1}\supset\dots\supset V^q=0
$$
and therefore,
$$
\aut_\calk(L)=\{f\in\aut(L),\,\,(f_*-\id_V)(V^i)\subset V^{i+1}\,\,\text{for all $i$}\,\}.
$$
We then apply Proposition \ref{exponencial} to conclude that $\derk L$ is precisely $\der L$ of Definition  \ref{derk} for the above filtration of $V$. Hence, by Proposition \ref{derkacomp},  $\derk L$ is a  well defined cdgl. Moreover, in view of Remark \ref{remarkexp}, the group $(\derk_0 L,*)$, endowed with the BCH product, is isomorphic to $\aut_\calk(L)$.

\end{rem}
 \begin{definition}\label{explogderG} Define
 $\derge L\subset\derk L$  by
$$
\derge_{\ge 1}L= \derr_{\ge1}L\quad\text{and}\quad \derge_0L=\{\text{$\theta\in\derr_0L$ such that  $D\theta=0$ and $e^{\theta}\in\aut_{\calge}(L)$}\}.
$$
Note again  that the group $(\derge_0 L,*)$ is identified, via Remark \ref{remarkexp}, with $\aut_{\calge}(L)$  which is a subgroup of $\aut_\calk(L)$. Thus,
 $\derge_0 L$ is a subgroup of $\derk_0 L$ for the BCH product.

\end{definition}

Then, the precise statement of Theorem \ref{mainre1intro}  is:

\begin{theorem}\label{mainre1} The maps
$$
L\stackrel{\ad}{\longrightarrow} \derge L \longrightarrow \derge L\timest sL
$$
constitute a cdgl fibration sequence whose realization is homotopy equivalent to the
classifying fibration
$$
X_\bq\longrightarrow B\aut^*_{G_\bq}(X_\bq)\longrightarrow B\aut_{G_\bq}(X_\bq)
$$
\end{theorem}
Here, $\ad$ denotes the adjoint operator and $(\derge L \timest sL,D)$ is the twisted product (\ref{slslrev1}), i.e.,   $sL$ is an abelian Lie algebra and
$$
Dsx=-sdx+\ad_x,\quad[\theta,sx]=(-1)^{|\theta|}s\theta(x),\quad x\in L,\quad \theta\in \derge L.
$$
Observe that combining this with Corollary \ref{cor1} we obtain:

\begin{corollary}\label{corofin} The realization of the cdgl sequence
\begin{equation}\label{larga}
\Hom(\overline\quic(L),L)\to\Hom(\quic(L),L)\to L\stackrel{\ad}{\to}\derge L\to\derge L\timest sL
\end{equation}
is of the homotopy type of
$$
\map^*(X_\bq,X_\bq)\to\map(X_\bq,X_\bq)\stackrel{\ev}{\to} X_\bq \to B\aut_{G_\bq}^*(X_\bq)\to  B\aut_{G_\bq}(X_\bq)
$$
In particular, restricting to the corresponding components in the above cdgl's produces a Lie model of the fibration sequence,
$$
\aut_{G_\bq}^*(X_\bq)\to\aut_{G_\bq}(X_\bq)\stackrel{\ev}{\to} X_\bq \to B\aut_{G_\bq}^*(X_\bq)\to  B\aut_{G_\bq}(X_\bq)
$$
\hfill$\square$
\end{corollary}

We now focus on a subgroup $\pi\subset \cale^*(X)$ of pointed homotopy classes of pointed equivalences which acts nilpotently on $\pi_*(X)$.
Then, by \cite[Theorem C]{drorza}, $B\aut_\pi^*(X)$ is a nilpotent complex and thus  $\pi$ is a nilpotent group. As in the homological case,  the rationalization $\pi_\bq$ of $\pi$ is similarly described as a subgroup of $\cale^*(X_\bq)$, this time with the aid of \cite[Theorem 3.3]{ma}. The analogs of Definitions \ref{explog} and  \ref{explogderG} are:

 \begin{definition}\label{explogpi} Define
$$
\Pi\subset \aut(L)/\sim,
$$
as the subgroup of homotopy classes of automorphisms of $L$ which corresponds to $\pi_\bq\subset\cale^*(X_\bq)$ under the isomorphism $\{L,L\}\cong\{X_\bq,X_\bq\}^*$ of Corollary \ref{corouno}.

 Consider also the subgroup of $\aut (L)$ given by
 $$
 \aut_{\Pi}(L)=\{\varphi\in\aut(L),\,[\varphi]\in \Pi\},
 $$
and define  $\derpi L\subset\derr L$  by,
$$
\derpi_{\ge 1}L= \derr_{\ge1}L\quad\text{and}\quad \derpi_0L=\{\text{$\theta\in\derr_0L$ such that  $D\theta=0$ and  $e^{\theta}\in\aut_{\Pi}(L)$}\}.
$$
As  $\pi_\bq$ acts nilpotently on $\pi_*(X_\bq)$ it does so in $H_*(X_\bq)$ (see \cite[Theorem 2.1]{hil}). Hence, $ \aut_{\Pi}(L)\subset  \aut_\calk(L)$ and thus, again by Remark \ref{remarkexp},   $\derpi_0 L$ is a well defined subgroup of $\derk_0 L$ for the BCH product.
\end{definition}

Finally, we will also assume that  $\pi_\bq$ is preserved by the action of $\pi_1(X_\bq)$, or equivalently, $\Pi$ is preserved by the action of $H_0(L)$. Then, Theorem \ref{mainre2intro} reads:

\begin{theorem}\label{mainre2} The twisted product
$$
L{\longrightarrow}L\timest \derpi L \longrightarrow \derpi L
$$
constitutes a cdgl fibration sequence whose realization is homotopy equivalent to the
classifying fibration
$$
X_\bq\longrightarrow Z_\bq\longrightarrow B\aut_{\pi_\bq}^*(X_\bq).
$$
\end{theorem}
Here $L\timest\derpi L$ is the restriction of the twisted product in (\ref{slslrev2}). That is, both terms are sub dgl's and $[\theta,x]= \theta(x)$ for any $\theta\in\derpi L$ and $x\in L$.
\begin{rem}\label{remark7} Note that, with the notation and nomenclature of the classical reference \cite[\S7]{may}, this theorem proves in particular that the geometric bar construction $B(*,\aut_\pi^*(X),X)$ has the rational homotopy type of the realization of $L\timest\derpi L$. Indeed in the universal classifying fibration (\ref{clasi}), $Z$ is precisely $B(*,\aut_\pi^*(X),X)$.
\end{rem}

\medskip

We point out that the hypothesis of Theorems \ref{mainre1} and \ref{mainre2} cannot be weakened. Let $K$ be a subgroup of $\cale^*(X_\bq)$ or $\cale(X_\bq)$ acting nilpotently on $\pi_*(X_\bq)$ or $H_*(X_\bq)$, and which does not necessarily arise as the rationalization of a nilpotent action on the homotopy or homology groups of $X$. Then, as shown in the following example, and even for $K$ abelian and $X$ simply connected, $B\aut_K^*(X_\bq)$ and $B\aut_K(X_\bq)$ do not lie in general in the image of the realization functor.

\begin{ex} Let $X=S^n\vee S^n$, with $n\ge 2$, whose minimal model is $M=(\lib(x,y),0)$ with $|x|=|y|=n-1$. Since $X_\bq$ is simply connected, free and pointed classes coincide. Consider the subgroup $K\subset\cale^*(X_\bq)=\cale(X_\bq)$   generated by the homotopy class of the automorphism $\varphi\colon M\stackrel{\cong}{\to} M$ given by $\varphi(x)=x+y$ and $\varphi(y)=y$. Obviously, $K$ acts nilpotently on both $\pi_*(X_\bq)$ and $H_*(X_\bq)$ as its central series vanishes at the second stage. Now, if $L$ is a cdgl for which $\langle L\rangle\simeq B\aut_K^*(X_\bq)$ it follows that $H_0(L)$, endowed with the BCH group structure, must coincide with $K$ which is in turn isomorphic to $\bz$.  However, a similar argument to the one in Proposition \ref{propo20} shows that this is not possible. The same applies to $B\aut_K(X_\bq)$.
\end{ex}

\begin{rem} Although all spaces involved in Theorems \ref{mainre1} and \ref{mainre2} are connected and nilpotent, it is important to stress that it is absolutely essential to work in $\catcdgl$ and not in the model category defined in \cite{nei} whose objects are  connected (not complete) dgl's. Among others, this is the main obstruction: fibrations in this model structure are surjective morphisms but there are fibrations of nilpotent complexes which cannot be modeled by  surjective  morphisms of connected dgl's.  A simple example of this kind of fibrations is given by the universal cover of the circle  $\br \to S^1$. However, it can be modeled by the following cdgl fibration in which we must allow elements of negative degrees in the source:
$$
\varphi\colon (\libc(x,y),d)\longrightarrow (\lib(x),0), \,\,\varphi(x)=x,\,\,\varphi(y)=0\esp{where} |x|=0,\quad dx=y.
$$
Note that the fibre of this morphism is given by,
$$
\ker\varphi=(\libc(a_n)_{n\ge0},d),\qquad a_n=\ad_x^n(y).
$$
Since each $a_n$ is a Maurer-Cartan element the realization of this cdgl is, as it should, a discrete countable space.  Other examples of this kind are given by the maps $X_\bq\to B\aut^*_{G_\bq}(X_\bq)$ and $X_\bq\to Z$ of Theorems \ref{mainre1} and \ref{mainre2} respectively, whose homotopy fibers $\aut_{G_\bq}(X_\bq)$ and $\aut_{\pi_\bq}^*(X_\bq)$ are  non-connected.
\end{rem}

\section{Modeling free and pointed classifying fibrations}\label{demostra}

In this section we complete  the proofs of Theorems \ref{mainre1} and \ref{mainre2}. We therefore adopt (and fix) the notation  in \S\ref{princi} and start with the following:

\begin{lemma}\label{welldefined} The maps
$$
L\stackrel{\ad}{\longrightarrow} \derge L \longrightarrow \derge L\timest sL,
$$
constitute a well defined cdgl sequence.
\end{lemma}

\begin{proof} We begin by proving that $\derge L$ is in fact a well defined cdgl. First, we see that $\derge_0 L$ is a complete Lie algebra. Up to this point, it is simply a subgroup of $\derk_0 L$ for the BCH product. However, recall from Remark \ref{derkcomplete} that $\derk_0 L$ coincides with $\der_0 L$ of Definition \ref{derk} which, by Lemma \ref{filderk}, is complete with respect to the filtration $\{\fil^n\}_{n\ge 1}$ in (\ref{filrev}). For each $n\ge 1$ denote $\fil^n_{\mathcal G}=\derge_0 L\cap \fil^n$ and note that this a subgroup of $\derge_0 L$ for the BCH product. Moreover, initially just as groups,
$$
\derge_0 L=\varprojlim_{n\ge 1}\derge_0 L/\fil^n_{\mathcal G},
$$
and thus $\derge_0 L$ is a pronilpotent group.

We check now that $\derge_0 L$ is a $0$-local group. Equivalently, we see that $\aut_{\mathcal G} (L)$ is $0$-local. First, for each $k\ge 1$, the map $\varphi\mapsto\varphi^k$ of $\aut_{\mathcal G}(L)$ is injective since this is a subgroup of  $\aut_{\mathcal K}(L)$ which is $0$-local. For the same reason, given $\varphi\in \aut_{\mathcal G} L$ and $k\ge 1$, there exists $\xi\in\aut_{\mathcal K}(L)$ such that $\xi^k=\varphi$. We finish by checking that $\xi$ lives in $\aut_{\mathcal G}(L)$. Since $\mathcal G$ is $0$-local there exists $\psi\in \aut_{\mathcal G}(L)$ such that $\psi^k\sim \varphi$. Thus, $\xi^k\sim \psi^k$ and, since $\mathcal K$ is $0$-local, it follows that $\xi\sim\psi$ so that $\xi\in\aut_{\mathcal G}(L)$.

 The pronilpotent and $0$-local character of $\derge_0 L$ let us apply Theorem \ref{cerrado} to conclude that this is a (Malcev) complete subgroup of $\derk_0 L$ and thus, by the Malcev category isomorphism (see \S3), $\derge_0 L$ is a complete Lie algebra with respect to the filtration $\fil^n_{\mathcal G}$.

 We check now that $D(\derr_1 L)\subset\derge_0 L$ by observing that,
$$
e^{D\eta}\sim\id L,\qquad\text{for any $\eta\in \derr_1 L$}.
$$
On the one hand, since the differential in $L=(\libc(V),d)$ is decomposable $(D\theta)(V)\subset \libc^{\ge 2}(V)$ and thus $e^{D\eta}$ is always well defined. On the other hand,  define a derivation $\tilde\eta$ in $\land (t,dt)\otimesc L$ by
$$\tilde \eta(t^n\otimes x)=t^{n+1}\otimes \eta(x), \quad \tilde\eta(t^ndt\otimes x)=-t^{n+1} dt\otimes \eta(x),\quad  n\geq 0, \quad x\in L,$$
and consider the cdgl morphism
$$\Phi=e^{D\tilde\eta}\circ \iota\colon L\to \land(t,dt)\otimesc L$$
where $\iota(x)=1\otimes x$. A straightforward computation shows that
$\varepsilon_0\circ\Phi=\id_L$ and $\varepsilon_1\circ \Phi= e^{D\eta}$ so that
$e^{D\eta}\sim \id_L$.

  All of the above shows that $\derge L$, and thus $\derge L\timest sL$, are sub dgl's of $\derk L$ and $\derk L\timest sL$ respectively, which are complete by Propositions \ref{derkacomp} and \ref{excusa}. Hence,  $\derge L$ and $\derge L\timest sL$ are also complete for the induced filtrations and all the objects in the statement are cdgl's.

We finish by checking that the image of the adjoint operator lies in $\derge L$. By definition, this amounts to say that, for each $x\in L_0$,  the homotopy class of $e^{\ad_x}$ lives in $\calge$. But this is trivial as, with the notation in Definition \ref{granaccion}, $[e^{\ad_x}]=[x]\bulletchi [\id_L]\in \calge $ since $\calge$ is closed by the action of $H_0(L)$.
\end{proof}

Next, consider the restriction of (\ref{clave}) to
\begin{equation}\label{fibprincipal}
0\to\Hom(\quic(L),L)\longrightarrow \Hom(\quic(L),L)\,\timest\,\derge L{\longrightarrow} \derge  L\to 0.
\end{equation}
Recall that  both terms in the middle  are sub dgl's and
$
[\theta,f]=\theta\circ f$ with $\theta\in\derge L$ and  $ f\in \Hom(\quic(L),L)$.

Fix the  MC element  $q\colon \quic(L)\to L$ of $\Hom(\quic(L),L)$
defined in Remark \ref{remark}.  With the notation in (\ref{varphifree}), $q=\overline{\id}_L$ as it represents  the homotopy class of the identity in $\{L,L\}/H_0(L)$ which, in turn, is identified with $\id_{X_\bq}$.

Note that (\ref{fibprincipal}) is a cdgl fibration sequence, so is its realization,
$$
\langle \Hom(\quic(L),L)\rangle\longrightarrow \langle \Hom(\quic(L),L)\,\timest\,\derge L\rangle {\longrightarrow} \langle \derge  L\rangle.
$$
We restrict this fibration, in which the base $ \langle \derge  L\rangle$ is connected, to the path component of the total space  $\langle \Hom(\quic(L),L)\,\timest\,\derge L\rangle$ containing the $0$-simplex $q$. We then obtain  another fibration
\begin{equation}\label{fibracion}
F\longrightarrow \langle \Hom(\quic(L),L)\,\timest\,\derge L\rangle^q\longrightarrow  \langle \derge  L\rangle.
\end{equation}
for which we identify its fiber and total space in the next results.

\begin{lemma}\label{lafibra} $F={\amalg}_{[\varphi]\in\calge/H_0(L)}\,\,\langle  \Hom(\quic(L),L)\rangle^{\overline\varphi}$.
\end{lemma}

\begin{proof} Observe that $F$ does not have to be connected as it is formed by  all the path components of $\langle \Hom(\quic(L),L)\rangle$ lying in $\langle \Hom(\quic(L),L)\,\timest\,\derge L\rangle^q$. That is, and  with notation in Remark \ref{remark},
$$
F={\amalg}_{\varphi}\,\,\langle  \Hom(\quic(L),L)\rangle^{\overline\varphi},
$$
where $\varphi\colon L\to L$ runs through the homotopy classes of cdgl morphisms such that\break   the MC element $\overline\varphi=\varphi q$ of $\Hom(\quic(L),L)$ is gauge related to  $q$ when considered as MC elements of  $\Hom(\quic(L),L)\,\timest\,\derge L$. That is, $\overline\varphi$ and $q$ represent the same element in $\widetilde\mc(\Hom(\quic(L),L)\,\timest\,\derge L)$.

Now, again in view of Remark \ref{remark}, $\overline\varphi$ and $q$ are already gauge related in $\Hom(\quic(L),L)$ if and only if $\varphi\sim\id_L$. If  this is not the case, but still $\overline\varphi$ is gauge related to $q$ in  $\Hom(\quic(L),L)\,\timest\,\derge L$, there must be a derivation,
 $$
 \theta\in \derge_0 L\esp{such that} \theta\gauge q=\overline\varphi.
 $$
That is,
$$
e^{\ad_\theta}(\overline\varphi)-\frac{e^{\ad_\theta}-1}{\ad_\theta}(D\theta)=q.
$$
Since $\theta$ is a cycle and $\overline\varphi=\varphi q$, this becomes,
$$
e^\theta=\varphi,\esp{that is,} \varphi\in\aut_\calge L.
$$
 Conversely, given  an automorphism $ \varphi\in\aut_\calge L$ consider $\theta=\log\varphi\in\der_0 L$ and one easily checks, as above, that $\theta\gauge q=\overline\varphi$.
\end{proof}

\begin{lemma}\label{elespaciototal}
$(\Hom(\quic(L),L)\,\timest\,\derge L)^q\simeq L$.
\end{lemma}

\begin{proof}  Consider the map,
 $$
  \gamma\colon L\longrightarrow(\Hom(\quic(L),L)\,\timest\,\derge L, D_q),\quad \gamma(x)=-x+\ad_x,\quad x\in L.
  $$
Note that in view of  (\ref{iso}), $\Hom(\quic(L),L)\cong \Hom(\overline\quic(L),L)\timest L$. Also, as shown in Lemma \ref{welldefined}, $\ad_x\in\derge L$ for any $x\in L$. Thus,   $\gamma$ is well defined.

On the other hand, recall from (\ref{iso}) and (\ref{fibprincipal}) that, in
$$
\Hom(\quic(L),L)\,\timest\,\derge L\cong \Hom(\overline\quic(L),L)\timest L\timest \derge L,
$$
the differential $D_q$ is defined as:
$$
\begin{aligned}
D_qf=&Df+[q,f],\quad f\in \Hom(\quic(L),L),\\
D_qx=&dx+[q,x]=dx-(-1)^{|x|}\ad_x\circ q,\quad x  \in L,\\
D_q\theta=&D\theta+[q,\theta]=D\theta-(-1)^{|\theta|}\theta\circ q,\quad\theta\in \derge L.\\
\end{aligned}
$$
Thus, a simple computation shows that $\gamma$ commutes with differentials. We then show that the component of $\gamma$ at $0$,
$$
\gamma\colon L\stackrel{\simeq}{\longrightarrow}(\Hom(\quic(L),L)\,\timest\,\derge L, D_q)^0=(\Hom(\quic(L),L)\,\timest\,\derge L)^q
$$
is a quasi-isomorphism.

We  recall again that $q=\overline\id_L$ with the notation in (\ref{varphifree}). In this particular case, the  cdgl isomorphism in (\ref{bueniso}) and the quasi-isomorphism of chain complexes  in (\ref{alfa}) become respectively,
$$
 \Gamma\colon \des\derr_{\alpha}(\quil\quic(L),L)
      \stackrel{\cong}{\longrightarrow}
      ( \Hom(\overline\quic(L),L),D_{q})\esp{and} \alpha^*\colon \derr L\stackrel{\simeq}{\longrightarrow}\derr_{\alpha}(\quil\quic(L'),L).
      $$
      The  composition $\Gamma\circ s^{-1}\alpha^*$ is then a quasi-isomorphism of chain complexes
      $$
      s^{-1} \derr L\stackrel{\simeq}{\longrightarrow} ( \Hom(\overline\quic(L),L),D_{q})
      $$
      which trivially extends, by the identity on $L\timest\derge$, to a quasi-isomorphism
      $$
      (s^{-1} \derr L\timest L\timest\derge L,\widetilde D)\stackrel{\simeq}{\longrightarrow}(\Hom(\overline\quic(L),L )\timest L\timest \derge L,D_q)
      $$
      where, on the left,

$$
\begin{aligned}
&\dtil s^{-1}\eta=-s^{-1}D\eta,\quad \eta\in \derr L,\\
&\dtil x=dx-s^{-1}\ad_x,\quad x  \in L,\\
&\dtil\theta=D\theta-s^{-1}\theta,\quad\theta\in \derge L.\\
\end{aligned}
$$
Obviously $\gamma$ factor through this map,
$$
\xymatrix{
 L\ar[r]^(.25){\widetilde\gamma}\ar[rd]_(.40)\gamma &  (s^{-1} \derr L\timest L\timest\derge L,\widetilde D) \ar[d]_\simeq \\
	& (\Hom(\overline\quic(L),L)\timest L\timest\derge L,D_q),}
$$
where, again, $\widetilde\gamma(x)=-x+\ad_x$. It is enough then to check that the component of $\widetilde\gamma$ at $0$,
$$
\widetilde\gamma\colon L\stackrel{\simeq}{\longrightarrow}(s^{-1} \derr L\timest L\timest\derge L,\widetilde D)^0
$$
is a quasi-isomorphism. Let $s^{-1}\eta+x+\theta$ be a $\dtil$-cycle in $s^{-1} \derr L\timest L\timest\derge L$
of non-negative degree. That is, $dx=D\theta=0$ and $D\eta+ad_x+\theta=0$. Then,
$$
s^{-1}\eta+x+\theta+ \widetilde\gamma(x)=s^{-1}\eta+\theta+\ad_x=-\dtil\eta.
$$
This shows that $H(\widetilde\gamma)$ is surjective. Finally, if $x\in L$ is a cycle for which $\widetilde\gamma(x)=\dtil(s^{-1}\eta+y+\theta)$, then $-dy=x$ and $H(\widetilde\gamma)$ is also injective.

Furthermore, observe that the projection,
\begin{equation}\label{rho}
\rho\colon (\Hom(\quic(L),L)\,\timest\,\derge L)^q\stackrel{\simeq}{\longrightarrow} L
\end{equation}
is a retraction of $\gamma$ and therefore, it is also a quasi-isomorphism.

\end{proof}

As a crucial consequence of the past results we get:

\begin{proposition}\label{propoimportante} There exists a fibration sequence
$$
F\stackrel{\xi}{\longrightarrow}\langle L\rangle\stackrel{\langle\ad\rangle}{\longrightarrow } \langle\derge L\rangle
$$
such that $\xi$ is homotopy equivalent to the evaluation fibration $\ev\colon \aut_{G_\bq}(X_\bq)\to X_\bq$.
\end{proposition}

\begin{proof}
With $\gamma$ as in Lemma \ref{elespaciototal}, consider the commutative diagram
$$
\xymatrix{
 (\Hom(\quic(L),L)\,\timest\,\derge L)^q\ar[r] &  \derge L. \\
	  L \ar[u]_\simeq^\gamma \ar[ru]_(.45)\ad &}
$$
whose realization
$$
\xymatrix{
 F\ar[r]&\langle \Hom(\quic(L),L)\,\timest\,\derge L\rangle^q\ar[r] & \langle \derge L\rangle. \\
	 &\langle L \rangle \ar[u]_\simeq^{\langle \gamma\rangle} \ar[ru]_(.45){\langle \ad\rangle} &}
$$
 exhibits a factorization of $\langle \ad\rangle $ as a weak equivalence followed by the fibration in (\ref{fibracion}). Next, define $\xi$ as the composition
$$
\xymatrix{
 F\ar[r]\ar[rd]_\xi&\langle \Hom(\quic(L),L)\,\timest\,\derge L\rangle^q\ar[d]^\simeq_{\langle\rho\rangle}  \\
	& \langle L\rangle}
$$
with $\rho$ as in (\ref{rho}).
But, by Lemma \ref{lafibra}, the restriction of $\xi$ to each path component of $F$,
$$
\xi\colon \langle  \Hom(\quic(L),L)\rangle^{\overline\varphi}\longrightarrow \langle L\rangle,\quad\varphi\in\aut_\calge,
$$
is precisely the realization of the projection,
$$
 \Hom(\quic(L),L)^{\overline\varphi}\longrightarrow L.
 $$
 However,  in view of Proposition \ref{escapa}, the realization of this cdgl morphism is homotopy equivalent to the evaluation map
$$
\ev\colon  \map_f(X_\bq,X_\bq)\longrightarrow X_\bq
$$
with $[f]\in G_\bq$ corresponding to the homotopy class of $\varphi$.  That is, we have a homotopy commutative diagram,
$$
\xymatrix{
 \aut_{G_\bq}(X_\bq)\ar[r]^(.60)\ev&X_\bq\\
F\ar[u]_\simeq\ar[r]^\xi&
\langle L\rangle.\ar[u]^\simeq\\}
$$
\end{proof}

\begin{rem}\label{remi} In the proof of Theorem \ref{mainre1} we will also need the following observation of  general nature: Let
$$
0\to \ker p\longrightarrow M\stackrel{p}{\longrightarrow} L\to 0
$$
be a fibration of connected cdgl's and
$$
\langle \ker p\rangle\longrightarrow\langle M\rangle\stackrel{\langle p\rangle}{\longrightarrow}\langle L\rangle$$
the realization fibration. Let $x\in L_0$, $y\in M_0$ with $p(y)=x$, and consider the automorphism
$
e^{\ad_y}\colon \ker p\stackrel{\cong}{\longrightarrow}\ker p$. Its realization
$$
\langle e^{\ad_y}\rangle\colon \langle \ker p\rangle \stackrel{\simeq}{\longrightarrow} \langle  \ker p\rangle$$
is precisely, up to homotopy, the homotopy equivalence of the fibre induced by the element $[x]\in\pi_1\langle L\rangle$. In view of the definition of the holonomy action for Kan fibrations (see for instance \cite[Corollary 7.11]{may2}), this is a simple exercise using just the definition of the realization functor and the identification of the elements $x$ and $y$ with the corresponding $1$-simplices of $\langle L\rangle$ and $\langle M\rangle$.

In other words, the natural map $\pi_1\langle L\rangle\to \cale\langle\ker p\rangle$ sends $[x]\in H_0(L)$ to the homotopy class of $\langle e^{\ad_y}\rangle$.
 \end{rem}

We are now able to complete the:

\begin{proof}[Proof of Theorem \ref{mainre1}]

We first check that
\begin{equation}\label{otrafibra}
L\stackrel{\ad}{\longrightarrow} \derge L \longrightarrow \derge L\timest sL,
\end{equation}
is a fibration sequence. For it, and
as in the classical, simply connected case \cite[\S VII.4(1)]{tan}, consider the twisted product of $L$ and $\derge L\timest sL$,
\begin{equation}\label{masfibra}
L\longrightarrow (L\timest \derge L\timest sL,\dtil)\longrightarrow \derge L\timest sL
\end{equation}
where, in the middle,
$$
[\theta,x]=\theta(x),\quad [sx,y]=0,\quad \dtil sx=-sdx-x+\ad_x,\quad \dtil\theta=D\theta,\quad x,y\in L,\quad\theta\in\derge L.
$$
In view of \S\ref{comple}, this is a cdgl fibration.

On the one hand, since the inclusion $i\colon \derge L\stackrel{\simeq}{\hookrightarrow} (L\timest \derge L\timest sL,\dtil)$ is a quasi-isomorphism, the diagram
\begin{equation}\label{hopull}
\xymatrix{(L\timest \derge L\timest sL,\dtil)\ar[rd]&\\
 \derge L\ar[u]_i^\simeq\ar[r] &  \derge L\timest sL}
\end{equation}
is a factorization of $\derge L\to \derge L\timest sL$ as a weak equivalence followed by a fibration.

On the other hand, the morphism
$$
\varrho\colon (L\timest \derge L\timest sL,\dtil)\stackrel{\simeq}{\longrightarrow}\derge L,\quad \varrho(x)=\ad_x,\,\,\varrho(\theta)=\theta,\,\,\varrho(sx)=0,\,\, x\in L,\,\, \theta\in\derge L,
$$
 is a left inverse of $i$ and thus, it is also a quasi-isomorphism which makes commutative this diagram,
$$
\xymatrix{&(L\timest \derge L\timest sL,\dtil)\ar[d]^\varrho_\simeq&\\
  L\ar[ru]\ar[r]_\ad &  \derge L.}
$$
These two facts exhibit (\ref{otrafibra}) and (\ref{masfibra}) as homotopy equivalent  fibration sequences.

Next we check that the realization of  (\ref{otrafibra}),
\begin{equation}\label{penul}
\langle L\rangle \stackrel{\langle \ad\rangle }{\longrightarrow}\langle \derge L \rangle\longrightarrow \langle \derge L\timest sL\rangle,
\end{equation}
is a fibration in $\fib_{G_\bq}(X_\bq, \langle \derge L\timest sL\rangle)$, see (\ref{clasefib}). For it,
 note that,
\begin{equation}\label{piuno}\pi_1\langle \derge L\timest sL\rangle=H_0(\derge L\timest sL)=H_0(\derge L)/\im H_0(\ad ).
\end{equation}
Hence, by Remark \ref{remi} and since $\exp(\derge L)=\aut_\calge (L)$, the image of the holonomy action $\pi_1\langle \derge L\timest sL\rangle\to\cale\langle L\rangle\cong\cale(X_\bq)$ is precisely $\calge/H_0(L)\cong G_\bq$ as required. Therefore, (\ref{penul})  is obtained as the pullback of the universal fibration over a certain map,
$$
\xymatrix{X_\bq\ar[r]&B\aut_{G_\bq}^*(X_\bq)\ar[r]&B\aut_{G_\bq}(X_\bq)\\
\langle L\rangle\ar[u]^\simeq\ar[r]_(.40){\langle\ad\rangle}&
\langle \derge L \rangle\ar[u]\ar[r]&
\langle \derge L \timest sL\rangle.\ar[u]\\}
$$
Finally, by Proposition \ref{propoimportante}, this diagram fits in a larger one
$$
\xymatrix{\aut_{G_\bq}(X_\bq)\ar[r]^(.59)\ev & X_\bq\ar[r]&B\aut_{G_\bq}^*(X_\bq)\ar[r]&B\aut_{G_\bq}(X_\bq)\\
F\ar[u]^\simeq\ar[r]_\xi&\langle L\rangle\ar[u]^\simeq\ar[r]_(.40){\langle\ad\rangle}&
\langle \derge L \rangle\ar[u]\ar[r]&
\langle \derge L \timest sL\rangle,\ar[u]\\}
$$
where, again, both arrows are fibration sequences. Thus, the last two vertical maps are homotopy equivalences, and the proof is complete.

\end{proof}

We now turn to the  proof of Theorem \ref{mainre2} for which the analogue of Lemma \ref{welldefined} reads:

\begin{lemma}\label{welldefined2} The maps
$$
L {\longrightarrow}L\timest \derpi L \longrightarrow \derpi L,
$$
constitute a well defined cdgl sequence.
\end{lemma}

\begin{proof} The same argument in the proof of Lemma \ref{welldefined} shows that $\derpi L$ is a subdgl of  $\derk L$. Therefore, by Propositions \ref{derkacomp} and \ref{excusa} respectively, $\derpi L$ and $ L\timest\derpi L$ are well defined cdgl's.
\end{proof}

\begin{proof}[Proof of Theorem \ref{mainre2}] To avoid excessive notation we denote by $\aut_{\pi_\bq}(X_\bq)$ the monoid $\aut_{\zeta(\pi_\bq)}(X_\bq)$ with $\zeta$ as in (\ref{surfun}). That is,
$$
\aut_{\pi_\bq}(X)=\{f\in\aut(X_\bq),\,\,[f]\in\zeta(\pi_\bq)\,\}.
$$
In view of Definitions \ref{gyg} and \ref{pyp}  note that,  since $\pi_\bq$ is preserved by the action of $\pi_1(X_\bq)$,
$$
\aut_{\zeta(\pi_\bq)}^*(X_\bq)=\aut_{\pi_\bq}^*(X_\bq).
$$
On the other hand, invoking again \cite[Theorem 2.1]{hil}, $\pi_\bq$ and therefore $\zeta(\pi_\bq)$ act nilpotently on $H_*(X_\bq)$.
Hence, we may apply Theorem \ref{mainre1} to obtain a homotopy commutative diagram,
$$
\xymatrix{B\aut_{\pi_\bq}^*(X_\bq)\ar[r]^p&B\aut_{\pi_\bq}(X_\bq)\\
\langle \derpi L \rangle\ar[u]^\simeq\ar[r]&
\langle \derpi L \timest sL\rangle\ar[u]_\simeq\\}
$$
where $p$ is the classifying fibration.
On the other hand, combining \cite[Proposition 7.8]{may} and \cite[Remark 1.2]{fuen}, we see that the classifying fibration $ Z_\bq\to B\aut_{\pi_\bq}^*(X_\bq)$ in Theorem \ref{mainre2} sits in a homotopy pullback,
$$
\xymatrix{B\aut_{\pi_\bq}^*(X_\bq)\ar[r]&B\aut_{\pi_\bq}(X_\bq)\\
Z_\bq\ar[u]\ar[r]&
B\aut_{\pi_\bq}^*(X_\bq).\ar[u]\\}
$$
Thus, since the realization functor preserves homotopy limits, the above square is homotopy equivalent to the realization of the cdgl homotopy pullback of the morphisms
$$
\derpi L\longrightarrow \derpi L \timest sL\longleftarrow \derpi L
$$
To compute this homotopy pullback factor any of the above morphisms as in (\ref{hopull}) and then take the pullback of the resulting fibration $(L\timest \derpi L\timest sL,\dtil)\to \derpi L\timest sL$ and $\derpi L\longrightarrow \derpi L \timest sL$,
$$
\xymatrix{
(L\timest \derpi L\timest sL,\dtil)\ar[r]&\derpi L \timest sL\\
L\timest\derpi L\ar[u]\ar[r]&
\derpi L.\ar[u]\\
}
$$
To conclude note that the realization of $L$, the fiber of the bottom morphism, must be homotopy equivalent a $X_\bq$, the homotopy fiber of the classifying fibration $Z\to
B\aut_{\pi_\bq}^*(X_\bq)$.

\end{proof}
\section{Some consequences and examples}\label{apli}

In this section, and unless explicitly stated otherwise, we adopt the notation in \S\ref{princi} and the assumptions in  Theorems \ref{mainre1} and \ref{mainre2}.

The first immediate consequence, which is Corollary \ref{corounointro}, is a description  of the rationalization of $\pi$ and $G$. In what follows and as usual, $H_0$ is considered as a group with the BCH product.

\begin{theorem}\label{apli1} $\pi_\bq\cong H_0(\derpi L)$ and $G_\bq\cong H_0(\derge L)/\im H_0(\ad)$.
\end{theorem}

\begin{proof}
Trivially, $\pi_\bq\cong\pi_1 B\aut_{\pi_\bq}(X_\bq)\cong \pi_1\langle\derpi L\rangle\cong  H_0(\derpi L)$. The second assertion is also immediate in view of (\ref{piuno}).
\end{proof}

\begin{rem} Observe that \cite[Proposition 12]{sal}, is the particular instance of Theorem \ref{apli1} for the  distinguished subgroups $\cale_H(X_\bq)\subset\cale(X_\bq)$ and $\cale_\pi(X_\bq)\subset\cale^*(X_\bq)$  of those classes that induce the identity on $H_*(X_\bq)$ and $\pi_*(X_\bq)$ respectively. In this case $\derge_0 L$ consists of derivations $\theta$ commuting with the differential and such that $e^\theta$ is an automorphism of $L=(\libc(V),d)$ inducing the identity on $V$. That is, $\theta(V)\in \libc^{\ge2}(V)$.
Analogously, $\derpi_0 L$ are those  derivations $\theta$ commuting with the differential and such that $e^\theta$ induces the identity on $H_*(L)$.
\end{rem}

Another immediate application concerns the homotopy nilpotency of $\aut^*_{\pi_\bq}(X_\bq)$ and $\aut_{G_\bq}(X_\bq)$. Recall that given  an $H$-group $Y$, its nilpotency $\nil Y$, is the least integer $n\le \infty$ for which the $(n+1)$th homotopy commutator of $Y$ is homotopically trivial.  On the other hand, for any dgl $L$ we denote by $\nil L$ the usual nilpotency index. Then:

\begin{proposition} $\nil \aut_{\pi_\bq}^*(X_\bq)=\nil H(\derpi L)$ and $\nil \aut_{G_ \bq}(X_\bq)=\nil H(\derge L\timest sL)$.
\end{proposition}

\begin{proof} By \cite[Theorem 3]{sal}, $\nil \aut_{\pi_\bq}^*(X_\bq)$ and $\nil \aut_{G_ \bq}(X_\bq)$ coincide, respectively, with the iterated Whitehead product length of $\pi_* B\aut_{\pi_\bq}^*(X_\bq)$ and $\pi_* B\aut_{G_ \bq}(X_\bq)$. Finally, apply  Theorems \ref{mainre1} and \ref{mainre2} taking into account that, as Lie algebras, $H_*(M)\cong\pi_{*+1}\langle M\rangle$ for any connected cdgl $M$ \cite[\S12.5.2]{bfmt0}.
\end{proof}

We now consider, again for $X$ nilpotent,  the monoid $\aut_1(X)$  of   self equivalences  homotopic to $\id_X$, which corresponds to  choosing $G\subset\cale(X)$  the trivial group. Note (see Definition \ref{gyg} and Remark \ref{explico}) that the submonoid $\aut_1^*(X)\subset\aut_1(X)$ is not connected in general. In this case,  Theorem  \ref{mainre1} asserts that the rational homotopy type of the classifying fibration
\begin{equation}\label{simplycon}
X\longrightarrow B\aut_1^*(X)\longrightarrow B\aut_1(X)
\end{equation}
is
 modeled by the cdgl fibration sequence
$$
L\stackrel{\ad}{\longrightarrow} \deri L\longrightarrow\deri L\timest sL,
$$
where $\cali\subset\aut L/\sim$ is the subgroup of homotopy classes of automorphisms for which the orbit group is trivial: $\cali/H_0(L)=\{1\}$. In this particular instance:

\begin{proposition}\label{aut1} \begin{itemize}
\item[(i)] $\deri L=\derr_{\ge 1}L\oplus R_0$ where $R_0=D(\derr_1 L)+\ad L_0$.
\item[(ii)] $\deri L\timest sL\simeq \widetilde{\derr L \timest sL}$ where, as usual, this is the $1$-connected cover of $\derr\timest sL$.
    \end{itemize}

\end{proposition}
    Note that
$
\bigl(\widetilde{\derr L \timest sL}\bigr)_{\ge 2}=(\derr L\timest sL)_{\ge 2}$ and  $$\bigl(\widetilde{\derr L \timest sL}\bigr)_{1}=\{(\theta,sx)\in \derr_1 L\times sL_0,\,\,D\theta=-\ad_x\,\}.
$$
\begin{proof} (i) By Theorem \ref{apli1}, $H_0(\deri L)/\im H_0(\ad)=0$, i.e. $H_0(\deri L)=\im H_0(\ad)$. That is, for any $x\in \deri_0 L$,  $[x]= [\ad y]$ for some $y\in L_0$. In other words $x=d\phi +\ad y$ for some $\phi\in \derr_1 L$. This translates to $\deri_0 L=R_0$.

(ii) In view of (i) this is straightforward. One can also argue that, as $B\aut_1(X)$ is simply connected, the realization of $\deri L\timest sL$ is of the homotopy type of its universal cover. To finish, recall from \S\ref{prelimi} that taking $1$-connected covers commutes with realization and note that the $1$-connected cover of $\deri L\timest sL$ is precisely $\widetilde{\derr L \timest sL}$.
\end{proof}

Interesting consequences of this results are the following in which $\aut_1(L)$ denote the group of automorphisms of $L$ homotopic to the identity.

\begin{corollary}\label{corfin1} For any connected, minimal cdgl $L$ the exponential restrict to a group isomorphism
$$
\exp\colon D(\derr_1 L)\stackrel{\cong}{\longrightarrow} \aut_1 (L).
$$
\end{corollary}
\begin{proof} The exponential $\exp\colon \deri_0 L\stackrel{\cong}{\to} \aut_{\mathcal I}(L)$ induces an isomorphism
$$
\exp\colon \deri_0 L/\ad L_0\stackrel{\cong}{\longrightarrow}\aut_{\mathcal I}(L)/e^{\ad L_0}.
$$
However, by (i) of Proposition \ref{aut1}, $\deri_0 L/\ad L_0=D(\derr_1 L)$. On the other hand,
via the action of $H_0(L)$ on $\deri_0 L$ (see Definition \ref{granaccion}), $\aut_{\mathcal I}(L)/e^{\ad L_0}$ is precisely $\aut_1(L)$.
\end{proof}

From this, we trivially see that the logarithm take homotopic automorphisms to homologus derivations:

\begin{corollary}\label{corfin2} Let $G\subset\aut(L)$ be a complete subgroup. Two automorphisms $f,g\in G$ are homotopic if and only if $\log(f)*\bigl(-\log(g)\bigr)=D\eta$ with $\eta\in \derr_1 L$.
\end{corollary}

\begin{proof} $f\sim g$ if and only if $fg^{-1}\sim\id_L$. By Corollary \ref{corfin1} this amounts to say that $\log(fg^{-1})=\log(f)*\bigl(-\log(g)\bigr)$ is in $D(\derr_1 L)$.
\end{proof}

\begin{rem} Observe that, when  $X$ is assumed to be simply connected, we recover from Proposition \ref{aut1} the classical result with which we started Section \ref{princi}. Indeed, in this case, $\aut_1^*(X)$ is also connected and thus (\ref{simplycon}) is a fibration sequence of simply connected spaces. On the other hand, $L$ is just the classical $1$-connected Quillen minimal model of $X$. Hence, applying proposition \ref{aut1} we deduce that,
$$
\deri L=\widetilde{\derr L},\qquad \deri L\timest sL=\widetilde{\derr L}\timest sL,
$$
and thus, (\ref{simplycon}) is modeled by
$$  L\stackrel{\ad}{\longrightarrow}\widetilde\derr L\longrightarrow \widetilde{\derr L} \timest sL.
$$
\end{rem}

Next, recall that given $G\subset\cale(X)$ and $\pi\subset\cale^*(X)$ there are   fibrations
\begin{equation}\label{pos2}
\aut_1(X)\longrightarrow \aut_G(X)\longrightarrow G\esp{and} \aut_1^*(X)\longrightarrow \aut_\pi^*(X)\longrightarrow \pi
\end{equation}
which extend to
\begin{equation}\label{pos}
B\aut_1(X)\longrightarrow B\aut_G(X)\longrightarrow BG\esp{and} B\aut_1^*(X)\longrightarrow B\aut_\pi^*(X)\longrightarrow B\pi.
\end{equation}
Here, according to  Definition \ref{pyp} and Remark \ref{explico}, $\aut_1^*(X)$ is the connected monoid of pointed maps based homotopic to $\id_X$.

Then, under the conditions of Theorems \ref{mainre1} and \ref{mainre2}, we obtain the following, which in particular proves Corollary \ref{corointrodos}:

\begin{theorem}\label{propopos} The fibration sequences
$$
B\aut_1(X_\bq)\to B\aut_{G_\bq}(X_\bq)\to BG_\bq\esp{and} B\aut_1^*(X_\bq)\to B\aut_{\pi_\bq}^*(X_\bq)\to B\pi_\bq
$$
have, respectively, the homotopy type of the realization of the cdgl fibrations,
$$
\widetilde{\derr L\timest sL}\hookrightarrow\derge L\timest sL\to (\derge L\timest sL) / (\widetilde{\derr L\timest sL})$$
and
$$ \widetilde\derr L\hookrightarrow\derpi L\to \derpi L/\widetilde\derr L.
$$
\end{theorem}

\begin{rem} A short computation let us observe that
$$
(\derge L\timest sL) / (\widetilde{\derr L\timest sL})= \derge_0 L\oplus R_1\quad\text{and }\quad \derpi L/\widetilde\derr L=\derpi_0 L\oplus S_1
$$
where $R_1$ and $S_1$ denote a complement of the cycles in degree $1$ of $\derge L\timest sL$ and $\derpi L$ respectively.
\end{rem}

\begin{proof} Note that the fibrations sequences in (\ref{pos}) can also be obtained  by fibring $B\aut_G(X)$ and $B\aut_\pi^*(X)$ over their first Postnikov stage. On the other hand, given $M$ a connected cdgl  and $n\ge 0$, consider the cdgl fibration
$$
M_{> n}\oplus Z_n\longrightarrow M\longrightarrow M/(M_{> n}\oplus Z_n)
$$
where $Z_n\subset M_n$ is the subspace of cycles. Then \cite[Proposition 12.43]{bfmt0}, the realization of this sequence is of the homotopy type of the fibration of $\langle M\rangle$ over its $n$th Postnikov stage. The result follows from choosing $n=1$ and  $M$  either $\derge L\timest sL$ or $\derpi L$. Indeed, for $M=\derge L\timest sL$, it follows that $M_{>1}\oplus Z_1=\widetilde{\derr L\timest sL}$. In the same way, for $M=\derpi L$, one checks that $M_{>1}\oplus Z_1=\widetilde\derr L$.
\end{proof}
The fibrations connecting (\ref{pos2}) and (\ref{pos}),
$$
\aut_G(X)\longrightarrow G\longrightarrow B\aut_1 (X)\esp{and} \aut_\pi^*(X)\longrightarrow \pi\longrightarrow B\aut_1^*(X),
$$
 can also be modeled.
For it, consider first the dgl twisted product,
$$
\Hom(\quic(L),L)\longrightarrow \Hom(\quic(L),L)\timest\derr L\timest sL\longrightarrow \derr L\timest sL,
$$
defined as follows: on the one hand,
the structure in $\Hom(\quic(L),L)\timest\derr L$ is the usual, i.e., that of (\ref{clave}). On the other hand, considering the isomorphism $\Hom(\quic(L),L)\cong \Hom(\overline\quic(L),L)\timest L$ in (\ref{iso}), define
$$
Dsx=x-sdx+\ad_x, \quad [sx,f]=0,\quad [\theta,sx]=(-1)^{|\theta}s\theta(x),\quad  [sx,y]=\frac{1}{2}s[x,y],
$$
 for $x,y\in L$, $f\in \Hom(\overline\quic(L),L)$, and $\theta\in\derr L$. Note that $s[x,y]=[sx,y]+(-1)^{|x|}[x,sy]$ which facilitates checking  that $D$ respects the bracket. As warned in \S\ref{comple} this may not be a cdgl sequence but the restriction to
 $$
\Hom(\quic(L),L)\longrightarrow \Hom(\quic(L),L)\timest\widetilde{\derr L\timest sL}\longrightarrow \widetilde{\derr L\timest sL}
$$
is so, as the left hand side is $1$-connected and thus, complete. Therefore, its  realization
$$
\langle \Hom(\quic(L),L)\rangle \longrightarrow \langle \Hom(\quic(L),L)\timest\widetilde{\derr L\timest sL}\rangle\longrightarrow \langle \widetilde{\derr L\timest sL}\rangle
$$ is a fibration and thus, by restricting the total space to the following particular set of path components we get a fibration,
\begin{equation}\label{recubrecu1}
F\longrightarrow {{\textstyle\coprod}}_{[\varphi]\in\calge/H_0(L)}\,\,\langle \Hom(\quic(L),L)\timest\widetilde{\derr L\timest sL}\rangle^{\overline\varphi} \longrightarrow \langle \widetilde{\derr L\timest sL}\rangle.
\end{equation}

Then we prove (cf. \cite[Theorem 1.1]{ber2}):

\begin{proposition}\label{recubre1} The  fibration (\ref{recubrecu1}) has the rational homotopy type of
$$
\aut_G(X)\longrightarrow G\longrightarrow B\aut_1 (X).
$$
\end{proposition}

\begin{proof} An analogous argument to that of Lemma \ref{lafibra} shows that
$$
F={{\textstyle\coprod}}_{[\varphi]\in\calge/H_0(L)}\,\,\langle  \Hom(\quic(L),L)\rangle^{\overline\varphi},
$$
which, in view of \S\ref{evalufibra}, has the homotopy type of $\aut_{G_\bq}(X_\bq)$.

Now, by construction, the total space of (\ref{recubrecu1}) has as many path components as the order of $G_\bq$. Finally, an analogous argument to that of Lemma \ref{elespaciototal} proves that, for each $\varphi$,
$$
\bigl( \Hom(\quic(L),L)\timest\widetilde{\derr L\timest sL}\bigr)^{\overline\varphi} \simeq 0.
$$
This shows that each of the components is homotopically trivial and the proposition follows.
\end{proof}

On the other hand, consider the restriction of the twisted product (\ref{clave}) to
$$
\Hom(\overline\quic(L),L)\longrightarrow \Hom(\overline\quic(L),L)\timest\widetilde\derr L\longrightarrow \widetilde\derr L
$$
which, in view of \S\ref{comple}, stays in $\catcdgl$. Its realization is a fibration
$$
\langle \Hom(\overline\quic(L),L)\rangle \longrightarrow \langle \Hom(\overline\quic(L),L)\timest\widetilde\derr L\rangle\longrightarrow \langle\widetilde\derr L\rangle
$$
and we restrict again to  certain path components of the total space to obtain another fibration,
\begin{equation}\label{recubrecu2}
\mathfrak{F}\longrightarrow {{\textstyle\coprod}}_{[\varphi]\in\Pi}\,\,\langle \Hom(\overline\quic(L),L)\timest\widetilde{\derr L}\rangle^{\overline\varphi} \longrightarrow \langle \widetilde{\derr L}\rangle.
\end{equation}
Similar arguments to those of Proposition \ref{recubre1} prove:
\begin{proposition}\label{recubre2} The  fibration (\ref{recubrecu2}) has the rational homotopy type of
$$
\aut_{\pi}^*(X)\longrightarrow \pi_\bq\longrightarrow B\aut_1^* (X).
$$
\hfill$\square$
\end{proposition}

We finish with two examples which cover a wide spectrum. We see how any finitely generated rational group of nilpotency index $n$ can be realized as a subgroup of self homotopy equivalences of the rationalization of a suitable finite complex, acting nilpotently on the homology of the complex with the same nilpotency index. The corresponding classifying group is explicitly described in terms of derivations.  On the other hand, any such group can also be realized as a subgroup of self homotopy equivalences  acting nilpotently on the homotopy groups of the complex with the same nilpotency index.

For that we first recall how to describe in simple terms the exponential and the logarithm  in the   finitely generated nilpotent case.  Denote by $U(n)$ the group of $n\times n$ {\em strictly triangular matrices} with rational entries, i.e., matrices $(m_{ij})$ with $m_{ij}= 0$ if $i\le j$.  In $U(n)$ we consider the Lie bracket given by commutators. On the other hand, denote by $T(n)$ the group of $n\times n$  {\em unitriangular matrices} also over $\bq$, that is, lower triangular matrices where all entries in the diagonal are 1. Then,  any finitely generated nilpotent $0$-local group and any finitely generated nilpotent Lie algebra  can be respectively embedded in $U(n)$ and $T(n)$ so that their logarithm and exponential maps are just the restriction of the classical bijections
$$
\xymatrix{U(n)\,\, \ar@<0.98ex>[r]_(.53){\scriptscriptstyle \cong}^-{\exp} &{\,T(n).} \ar@<0.95ex>[l]^-{\log} }
$$

\begin{ex} Let $M$ be a finitely generated nilpotent Lie algebra, with nilpotency index less than or equal to $n$, and  concentrated in degree $0$. By the above observation, $M$ can be embedded in $T(n)$. Consider the finite nilpotent complex $X=\vee_{j=1}^n S^m$ with $m>1$, whose minimal model is $L=(\lib(y_1,\dots,y_n),0)$ where $|y_j|=m-1$ for all $j$. Note that, for degree reasons, any derivation $\theta\in\derr_0 L$  is necessarily  of the form $$
\theta(y_j)=\sum_{k=1}^n\lambda_{j k}\,y_k,\qquad\lambda_{j k}\in\bq.
$$
That is, $\theta$ is identified with the matrix $(\lambda_{j k})$ considering
\begin{equation}\label{matriz}
\begin{pmatrix}
  \theta(y_1) \\
  \vdots \\
\theta(y_n)
\end{pmatrix}=\theta \begin{pmatrix}
  y_1 \\
  \vdots \\
y_n
\end{pmatrix}.
\end{equation}
Hence, $M$ is a sub Lie algebra of $\derr_0(L)$ which, being in $T(n)$, it determines a decreasing filtration of length at most $n$ of $V=\Span\{y_1,\dots,y_n\}\cong s^{-1}\widetilde H_*(X_\bq)$,
\begin{equation}\label{filfinal} V=V^0\supset\dots\supset V^i\supset V^{i+1}\supset\dots\supset V^q=0,\end{equation}
where $V^i=\theta(V^{i-1})$, $i\ge 1$.

On the  other hand, if we denote by ${\mathcal G}=\exp(M)$, we may also identify any matrix $\varphi=e^\theta\in{\mathcal G}$ with an automorphism of $L$ as in (\ref{matriz}):
$$
\varphi\colon L\stackrel{\cong}{\longrightarrow} L,\qquad \begin{pmatrix}
  \varphi(y_1) \\
  \vdots \\
\varphi(y_n)
\end{pmatrix}=\varphi \begin{pmatrix}
  y_1 \\
  \vdots \\
y_n
\end{pmatrix}.
$$
Note that $\aut(L)=\aut(L)/\sim$ so that  ${\mathcal G}=\aut_{\mathcal G}(L)$ is identified with the rational subgroup $G$ of $\cale(X_\bq)$ whose action in $H_*(X_\bq)$ produces, as central series, the suspension of the filtration (\ref{filfinal}) of $s^{-1}\widetilde H_*(X_\bq)$ (see Definitions \ref{K} and \ref{explog}). That is, $G$ acts nilpotently on $H_*(X_\bq)$.

 In other terms, $M=\derge_0 L$ and, by Theorem \ref{mainre1}, the sequence
 $$
L\stackrel{\ad}{\longrightarrow} \derge L \longrightarrow \derge L\timest sL
$$
is a Lie model of the
classifying fibration
$$
X_\bq\longrightarrow B\aut^*_{G}(X_\bq)\longrightarrow B\aut_{G}(X_\bq).
$$
Proceeding exactly in the converse direction, any  nilpotent rational group $G$ of $\cale(X_\bq)$ can be identified with a rational subgroup of $U(n)$ and thus, with a  subgroup $\mathcal G=\aut_{\mathcal G}(L)$ of automorphisms of $L$.  In turn, this determines  the nilpotent Lie algebra $M=\log({\mathcal G})\subset T(n)$ so that $\derge_0 L=M$.

 Just to illustrate how to proceed in a particular instance, choose for example
 $$
M=T(3)=\left\{\begin{pmatrix}
                              0 & 0 & 0 \\
                              \alpha & 0 & 0 \\
                              \beta & \gamma & 0
                            \end{pmatrix},\,\,\alpha,\beta,\gamma\in\bq\right\}
                             $$  and let
                              $X=S^m\vee S^m\vee S^m$, with $m> 1$ whose minimal model is $L=(\lib(y_1,y_2,y_3),0)$. Every $\theta\in M$ acts in $s^{-1}\widetilde H_*(X)$ by $\theta(y_1)=0$, $\theta(y_2)=\alpha y_1$ and $\theta(y_3)=\beta y_1+\gamma y_2$. This determines the filtration
                              $$
                              s^{-1}\widetilde H_*(X)=\Span\{y_1,y_2,y_3\}\supset\Span\{y_1,y_2\}\supset\Span\{y_1\}\supset 0.
                              $$
Then,  the group $\mathcal G=\aut_{\mathcal G}(L)=\exp M=U(3)$ is given by the automorphisms
$$
\{\varphi\colon L\stackrel{\cong}{\to} L,\,\,\varphi(y_1)=\lambda y_1,\,\,\varphi(y_2)=\lambda y_1+ y_2,\,\,\varphi(y_3)=\mu y_1+ \rho y_2+ y_3 \,\,\,\lambda,\mu,\rho \in\bq\},
$$
and is identified with a subgroup $G\subset \cale(X_\bq)$ which acts nilpotently on $H_*(X_\bq)$ with nilpotency index $3$. Note that  $M=\derge_0 L$.

\end{ex}
\setcounter{footnote}{1}
\begin{ex}
We consider now the dual setting: let $\Pi$ be a finitely generated, nilpotent, rational  group, with nilpotency index less than or equal to $n\ge 1$.   By the observation above, $\Pi$ is embedded in $U(n)$. Consider the  complex $Y=\prod_{j=1}^n S^m$ with $m\ge 1$ odd whose Sullivan minimal model is $A=(\land V,0)$. with $V=\Span\{x_1,\dots,x_n\}$ concentrated in degree $m$. Then, any map $Y_\bq\to Y_\bq$ is necessarily modeled by a cdga morphism $\psi\colon A\to A$ which might be identified with an $n\times n$ matrix  by imposing
$$
 \begin{pmatrix}
  \psi(x_1) \\
  \vdots \\
\psi(y_n)
\end{pmatrix}=\psi \begin{pmatrix}
  x_1 \\
  \vdots \\
x_n
\end{pmatrix}.
$$
Note also that  $\aut(A)=\aut(A)/\sim$, that is,  $\aut^*(Y_\bq)=\aut^*(Y_\bq)/\sim$ and thus,
the group $\Pi$ is identified with a subgroup $\pi$ of $\cale^*(Y_\bq)$ acting nilpotently on $\pi_*(Y_\bq)\cong V$ and thus, acting also nilpotently on $H_*(Y_\bq)\cong A^\sharp$. Observe that the action of $\pi_1(Y_\bq)$ (which is trivial or abelian if $Y$ is a torus) on $\{Y_\bq,Y_\bq\}^*$ is trivial and thus $\cale^*(Y_\bq)=\cale(Y_\bq)$.

By \cite[Theorem 10.2]{bfmt0} the cdgl
$$
\widehat\quil(A^\sharp)
$$
is a Lie model of $Y$ which turns out to be minimal. Indeed, in view of the explicit definition of $\quil$ in (\ref{quillenlc}), denote
$$
y_{i_1\dots i_s}=s^{-1}(x_{i_1}\dots x_{i_s})^\#,\esp{with} 1\le i_1<\dots< i_s\le n\quad 1\le s\le n,
$$
and observe that
$$
\widehat\quil(A^\sharp)=(\libc(y_{i_1\dots i_s}),d),
$$
where the differential is quadratic and  given as follows: fixed integers $1\leq i_1< \dots <i_s\leq n$ and let $E$ be the set of decompositions of   $\{i_1, \dots , i_s\}$ in two disjoint tuples,  $\{j_1, \dots , j_p\}$ and $\{k_1, \dots , k_q\}$, with $j_1<\dots < j_p$ and $k_1<\dots < k_q$. Then,

$$d(y_{i_1\dots i_s}) = \frac{1}{2}\sum_E \varepsilon_E [y_{j_1\dots j_p}, y_{k_1\dots k_q}],$$
where $\varepsilon_E$ denotes the sign of the permutation
$$i_1, \dots , i_s \mapsto j_1, \dots , j_p, k_1, \dots , k_q.$$
Remark that, whenever $Y$ is not the torus, i.e, $m>1$, then $\widehat\quil(A^\sharp)=\quil(A^\sharp)$ as no generator of degree $0$ appears. Call $L=\widehat\quil(A^\sharp)$ and  note that the map $\psi\mapsto\text{homotopy class of}\,\,\widehat\quil(\psi)$ defines a group isomorphism
\begin{equation}\label{finfin2}
\aut(A)\cong\aut(L)/\sim
\end{equation}
which exhibits $\Pi$ as a subgroup of $\aut(L)/\sim$. We may then consider $\aut_{\Pi}(L)$, which is a complete, possibly non nilpotent group, and thus different from $\Pi$ in general.  By Theorem \ref{mainre2} the corresponding cdgl  $\derpi(L)$ is a Lie model of $B\aut_{\pi}(Y_\bq)$.

   For instance, choose
\begin{equation}\label{finfin}
\Pi=U(2)=\left\{\begin{pmatrix}
                              1 & 0 \\
                              \lambda & 1
                            \end{pmatrix},\,\,\lambda\in\bq\right\}
                           \end{equation}
which is isomorphic to $\bq$, and corresponds to a group $\pi$ of self homotopy equivalences of the rational torus $T_\bq=S^1_\bq\times S^1_\bq$   whose Sullivan minimal model is $A=(\land(x_1,x_2),0)$. As explained before, the action of $\Pi$ on $\Span\{x_1,x_2\}$, of nilpotency index $2$, corresponds to an action of $\pi$ on $\pi_*(T_\bq)$ which has the following central series:
                            $$
                            \pi_*(T_\bq)\cong\Span\{x_1,x_2\}\supset \Span\{x_1\}\supset 0.
                            $$
        As noted, the minimal Lie model of $T_\bq$ is the cdgl
        $$
        L=\widehat\quil(A^\sharp)=(\libc(y_1,y_2,y_{12}),d),\quad |y_1|=|y_2|=0,\quad |y_{12}|=1,\quad dy_{12}=[y_1,y_2].
        $$
        Another easy computation shows that for every $\psi\in \Pi$  as in (\ref{finfin}), $\widehat\quil (\psi)$ is the automorphism of $L$ given by the $3\times 3$ matrix
\begin{equation}\label{finfinfin}
        B=\begin{pmatrix}
          1 & 0 & 0 \\
          \lambda & 1 & 0 \\
          0 & 0 & 1
        \end{pmatrix}.\qquad\text{That is,}\qquad  \begin{pmatrix}
 \widehat\quil (\psi)(y_1) \\
 \widehat\quil (\psi)(y_2) \\
\widehat\quil (\psi)(y_{12})
\end{pmatrix}=B \begin{pmatrix}
  y_1 \\
  y_2 \\
y_{12}
\end{pmatrix}
\end{equation}
\end{ex}
Now, in view of (\ref{finfin2}), any automorphism in $\aut_\Pi(L)$ is homotopic to $\widehat\quil (\psi)$ for some $\psi\in\Pi$. Hence, by Corollary \ref{corfin2}, and since the differential is trivial in $L$,
$$
\log\bigl(\aut_\Pi(L)\bigr)=\log(\Pi)+D(\derr_1L)=\log(\Pi).
$$
Note that for any matrix $B$ as in (\ref{finfinfin}),
$$
\log(B)=\begin{pmatrix}
          0 & 0 & 0 \\
          \lambda & 0 & 0 \\
          0 & 0 & 0
        \end{pmatrix}.
        $$
Thus, the Lie algebra $\log(\Pi)$ is identified to the $3\times 3$ matrices of this sort which, in turn, define the corresponding derivations of degree $0$ in $L$. Denote $M=\log(\Pi)$ and observe that the induced filtration of $M$ on $\Span\{y_1,y_2,y_{12}\}$ is
        $$
         \Span\{y_1,y_2,y_{12}\}\supset \Span\{y_1\}\supset 0.
         $$
         Finally,  by Theorem \ref{mainre2},
        $\derr_M L$ is a Lie model of $B\aut_{\pi}^*(T_\bq)$.

\bigskip
\bigskip\bigskip
\noindent {\sc Institut de Math\'ematiques et Physique, Universit\'e Catholique de Louvain, Chemin du Cyclotron 2,
1348 Louvain-la-Neuve,
         Belgique}.

\noindent\texttt{yves.felix@uclouvain.be}

\bigskip

\noindent{\sc Departamento de \'Algebra, Geometr\'{\i}a y Topolog\'{\i}a, Universidad de M\'alaga, 29080 M\'alaga, Spain.}

\noindent
\texttt{m\_fuentes@uma.es, aniceto@uma.es}

 \end{document}